\documentclass[11pt,reqno]{amsart}
\usepackage{amsmath,amsthm,amsfonts,amssymb,amscd,amsbsy}
\usepackage[utf8]{inputenc}
\usepackage{array,colortbl}
\usepackage{tikz}
\usetikzlibrary{cd,arrows,babel}
\usepackage{stmaryrd}
\usepackage{mathtools}
\usepackage{graphicx}
\usepackage[english]{babel}
\usepackage{fullpage}

\usepackage{graphicx}
\usepackage[colorlinks=true, allcolors=blue]{hyperref}

\newcommand{\N}{\mathbb{N}}
\newcommand{\ideal}{\mathcal{I}}
\newcommand{\idealj}{\mathcal{J}}
\newcommand\fin{{\sf Fin}}
\newcommand{\Q}{\mathbb{Q}}

\newtheorem{theorem}{Theorem}[section]
\newtheorem{lemma}[theorem]{Lemma}

\newtheorem{proposition}[theorem]{Proposition}
\newtheorem{corollary}[theorem]{Corollary}
\newtheorem{claim}[theorem]{Claim}

\theoremstyle{definition}
\newtheorem{remark}[theorem]{Remark}

\newtheorem{definition}[theorem]{Definition}

\newtheorem{question}[theorem]{Question}

\title{Banach spaces of $\mathcal I$-convergent sequences}
\author{Michael A. Rincón-Villamizar, Carlos  Uzcátegui Aylwin}

\begin{document}
\maketitle

\begin{abstract}
We study the space $c_{0,\mathcal{I}}$ of all bounded sequences $(x_n)$ that $\ideal$-converge to $0$, endowed with the sup norm, where $\ideal$ is an ideal of subsets of $\N$. We show that two such spaces, $c_{0,\ideal}$ and $c_{0,\idealj}$, are isometric exactly when the ideals $\ideal$ and $\idealj$ are isomorphic. Additionally, we analyze the connection of the well-known Kat\v{e}tov pre-order $\leq_K$ on ideals with some properties of the space $c_{0,\ideal}$. For instance, we show that $\mathcal{I}\leq_K\mathcal{J}$ exactly when there is a (not necessarily onto) Banach lattice isometry from $c_{0,\mathcal{I}}$ to $c_{0,\mathcal{J}}$, satisfying some additional conditions.

We present some lattice-theoretic properties of $c_{0,\ideal}$, particularly demonstrating that every closed ideal of $\ell_\infty$ is equal to $c_{0,\ideal}$ for some ideal $\ideal$ on $\N$. We also show that certain classical Banach spaces are isometric to $c_{0,\ideal}$ for some ideal $\ideal$, such as the spaces $\ell_\infty(c_0)$ and $c_0(\ell_\infty)$. Finally, we provide several examples of ideals for which $c_{0,\mathcal{I}}$ is not a Grothendieck space.
\end{abstract}

\section{Introduction}
An ideal on $\mathbb{N}$ is a collection $\ideal$  of subsets of $\mathbb{N}$ closed under finite unions and taking subsets of its elements. A sequence $(x_n)$ in $\mathbb{R}$ is said to be $\mathcal{I}$-convergent to $x\in\mathbb{R}$, denoted as $\mathcal{I}$-$\lim x_n=x$, if for each $\varepsilon>0$, the set $\{n\in\mathbb{N}\,\colon\,|x_n-x|\geq\varepsilon\}$ belongs to $\mathcal{I}$. When $\mathcal{I}$ is $\fin$, the ideal of finite subsets of $\mathbb{N}$, we have the classical convergence in $\mathbb{R}$. The $\mathcal{I}$-convergence was introduced in \cite{k-s-w}, although many authors had already studied this concept in particular cases and in different contexts (see, for instance, \cite{balcerzak,fast,filipow,k-m-s}). The main goal of this paper is to study the following space:
\begin{gather*}
c_{0,\mathcal{I}}=\{(x_n)\in\ell_\infty\,\colon\,\mathcal{I}-\lim x_n=0\}.
\end{gather*}
This space has recently received some attention (see, for instance, \cite{kania, Leonetti2018}). It is known that $c_{0,\mathcal{I}}$ is a closed subspace (also, an ideal) of $\ell_\infty$ and it is isometric to $C_0(U_\mathcal{I})$ for some open set $U_\mathcal{I}$ of $\beta\mathbb{N}$ (see \cite{kania} or Proposition \ref{identification c_0,I and C_0(U_I)}). Two extreme examples are worth keeping in mind. On the one hand,  as we mentioned before, $c_{0,\fin}$ is exactly $c_0$. On the other hand, if $\mathcal{I}$ is the trivial ideal $\mathcal{P}(\mathbb{N})$, we obtain the whole space $\ell_\infty$. 

A natural question is to determine when two such spaces are isomorphic (isometric). As a consequence of the Banach-Stone theorem, $c_{0,\mathcal{I}}$ and $c_{0,\mathcal{J}}$ are isometric exactly when $\mathcal{I}$ and $\mathcal{J}$ are isomorphic as ideals. 
However, the same does not hold for isomorphism. Indeed, it is known that if $\mathcal{I}$ is a maximal ideal, then $c_{0, \mathcal{I}}$ is complemented in $\ell_\infty$ and thus isomorphic to $\ell_\infty$ (see \cite{Leonetti2018}); but  $\ell_\infty$ is equal to $c_{0,\mathcal{P}(\mathbb{N})}$ and a maximal ideal is not isomorphic to $\mathcal{P}(\mathbb{N})$.


Ideals on countable sets have been studied for a long time, and several ways of comparing them have been investigated (see the surveys \cite{Hrusak2011,uzcasurvey}). We are particularly interested in the Kat\v{e}tov pre-order (see, for instance, \cite{Hrusak2017}). Given two ideals $\mathcal{I}$ and $\mathcal{J}$ on two countable sets $X$ and $Y$, respectively, we say that $\mathcal{I}$ is Kat\v{e}tov below $\idealj$, denoted as $\mathcal{I}\leq_K\mathcal{J}$, if there is $f:Y\to X$ such that $f^{-1}(A)\in \mathcal{J}$ for all $A\in \mathcal{I}$. Following the ideas behind a proof of Holsztynki's theorem \cite{rincon-villamizar}, we show that $\mathcal{I}\leq_K\mathcal{J}$ exactly when there is a (non-necessarily onto) Banach lattice isometry from $c_{0,\mathcal{I}}$ to $c_{0,\mathcal{J}}$ satisfying some additional conditions (see Theorem \ref{HT}).

 We study some lattice-theoretic properties of the spaces $c_{0,\mathcal{I}}$ and, in particular, we show that any closed ideal of $\ell_\infty$ is equal to $c_{0,\mathcal{I}}$ for some ideal $\mathcal{I}$ on $\mathbb{N}$. We found an interesting connection between the ideal theoretic notion of orthogonal ideals \cite{GU2018, To} and the notion of $c_0$-disjoint subspaces of $\ell_\infty$ (see section \ref{latticesProp}).

We show how to represent some classical Banach spaces as a space of the form $c_{0,\mathcal{I}}$ for some ideal $\mathcal{I}$ on $\mathbb{N}$. For example, we found ideals providing an  isometric representation of  the spaces $\ell_\infty(c_0)$ and $c_0(\ell_\infty)$. These two spaces have been recently shown to be non-isomorphic (see \cite{cembranos-mendoza2010}). We present a  different argument  to show that they are not isometric. We also found an uncountable collection of pairwise non-isometric spaces of the form $c_{0,\mathcal{I}}$. In \cite[Problem 9]{gonzalez-kania}, it was asked about ideals $\mathcal{I}$ such that $c_{0,\mathcal{I}}$ is Grothendieck; we present several non-Grothendieck spaces.

\section{Preliminaries}

\subsection{Banach spaces and Banach lattices.} 

We will use standard terminology and notation for Banach lattices and Banach space theory. For unexplained definitions and notations, we refer to \cite{AK,luxe}. All Banach lattices analyzed here are assumed to be real. $B_X$ stands for the closed unit ball of $X$. The positive cone of a Banach lattice $X$ is denoted by $X^+$. A sublattice $Y$ of a Banach lattice $X$ is an \textit{ideal} if $x\in Y$ whenever $|x|\leq|y|$ for some $y\in Y$. The closed ideal generated by a subset $A$  of $X$ is denoted by $\langle A\rangle$. If $X$ and $Y$ are isomorphic Banach spaces, we write $X\sim Y$.


If $X$ and $Y$ are Banach lattices, a linear operator $T\colon X\to Y$ is called a \textit{Banach lattice isomorphism} if $T$ is a Banach isomorphism such that $T(x\wedge y)=Tx\wedge Ty$ for all $x,y\in X$. Furthermore, if $T$ is an  isometry, we say that $T$ is a \textit{Banach lattice isometry}. Finally, when $T$ is not necessarily onto we will called it an {\em into} Banach lattice isomorphism (isometry, respectively).

Now, we introduce some notation. Let $(X_j)_{j\in\mathbb N}$ be a family of Banach spaces.
\begin{enumerate}
    \item $\ell_\infty((X_j)_{j\in\mathbb N})$ denotes the $\ell_\infty$-sum of $(X_j)_{j\in\mathbb N}$, that is, the Banach space of all bounded sequences $\displaystyle(x_j)\in\prod_jX_j$  endowed with the norm $\|\cdot\|$
given by $\displaystyle\|(x_j)\|=\sup_j\|x_j\|$. 

\item $c_0((X_j)_{j\in\mathbb N})$ is the $c_0$-sum of $(X_j)_{j\in\mathbb N}$, that is, the closed subspace of $\ell_\infty((X_j)_{j\in\mathbb N})$ consisting of all null sequences, i.e, $\displaystyle\lim_j\|x_j\|=0$.
\end{enumerate}
When $X=X_j$ for all $j\in\mathbb N$, these spaces are denoted by  
$\ell_\infty(X)$ and $c_0(X)$, respectively. Finally, for $X=\mathbb R$, $\ell_\infty(X)$ and $c_0(X)$ correspond to $\ell_\infty$ and $c_0$, respectively.

\begin{remark}\label{isomorphisms of l_oo(c_0)}
    It is easy to see that if each $X_j$ is isometric (isomorphic) to $X$, then
$\ell_\infty((X_j)_{j\in\mathbb N})$ ($c_0((X_j)_{j\in\mathbb N})$) is isometric (isomorphic, respectively) to $\ell_\infty(X)$ ($c_0(X)$, respectively).
\end{remark}

\subsection{Ideals.} Recall that an ideal $\mathcal I$ on 
 a set $X$ is a collection of subsets of $X$ satisfying:
\begin{enumerate}
    \item $\emptyset\in\mathcal I$;
    \item If $A\subset B$ and $B\in\mathcal I$, then $A\in\mathcal I$;
    \item If $A,B\in\mathcal I$, then $A\cup B\in\mathcal I$.
\end{enumerate}
We always assume that every finite subset of $X$ belongs to  $\mathcal I$. The dual filter of an ideal $\mathcal I$ is denoted by $\mathcal I^*$ and consists of all sets of the form $X\setminus A$ for some $A\in \ideal$. If $\mathcal A$ is a family of subsets of $X$, $\mathcal I(\mathcal A)$ denotes the ideal generated by $\mathcal A$ which consists of all subsets of finite unions of sets from $\mathcal{A}$. If $A\subset X$ and  $\mathcal I$ is an ideal on $X$, we denote the restriction of $\ideal$ to $A$ by $\mathcal I\restriction A=\{A\cap B\,\colon B\in\mathcal I\}$ which is an ideal on $A$. If $\mathcal I$ and $\mathcal J$ are ideals on $X$, the ideal $\mathcal I\sqcup\mathcal J$ is defined as $\mathcal I\sqcup\mathcal J:=\{A\cup B\,\colon\,A\in\mathcal I,B\in\mathcal J\}.$

Two ideals $\mathcal I$ and $\mathcal J$ on $\mathbb N$ are \textit{isomorphic} if there is a bijection (called isomorphism) $f\colon\mathbb N\to\mathbb N$ such that $A\in\mathcal J$ iff $f^{-1}(A)\in\mathcal I$. We recall the {\em Kat\v{e}tov pre-order} on ideals. If $\mathcal I$ and $\mathcal J$ are ideals on $X$ and $Y$, respectively, we write $\mathcal I\leq_K\mathcal J$ if there is a function $f\colon Y\to X$ (called a {\em Kat\v{e}tov reduction}) such that $f^{-1}(A)\in\mathcal J$ for all $A\in\mathcal I$.  In most of the cases, the function $f$ can be assumed  to be onto.

\begin{proposition}\cite{BFMS2013}
\label{katetoveOnto}
Let $\mathcal I$ and $\mathcal J$ be ideals on $X$ and $Y$, respectively. Suppose $\idealj$ contains an infinite set and  $\ideal\leq_K\idealj$, then there is an onto map  $f\colon Y\to X$ such that $f^{-1}(A)\in\mathcal J$ for all $A\in\mathcal I$. 
\end{proposition}

In general, Kat\v{e}tov reductions are not bijective but this naturally suggests a variant of $\leq_K$ which  was studied in \cite{BFMS2013}. They defined $\ideal\sqsubseteq \idealj$ if there is a bijective Kat\v{e}tov reduction from $\ideal$ to $\idealj$. In other words, $\ideal\sqsubseteq \idealj$ if there is an ideal $\idealj'\subseteq \idealj$ such that $\ideal$ is isomorphic to $\idealj'$. We will use this variant in the sequel.

The \textit{Fubini product}  is defined as follows.  For two ideals $\ideal$ and $\mathcal J$ on $\N$,  $\ideal\times \mathcal J$ is the ideal on $\N\times\N $ given  by:
\begin{gather*}
    A\in\mathcal I\times\mathcal J\Longleftrightarrow 
    \{n\in\mathbb N\,\colon\,\{m\in\mathbb N\,\colon\,(n,m)\in A\}\not\in\mathcal J\}\in\mathcal I.
\end{gather*}

Another concept that will be used in the sequel is the following. Given a collection $\mathcal{A}$ of subsets of $\N$, the {\em orthogonal} of $\mathcal{A}$ (see \cite{GU2018,To}) is the following family of sets
\[
{\mathcal{A}}^{\perp}=\{B\subseteq \N : (\forall A\in {\mathcal{A}})(A\cap B
\textrm{ is finite}) \}.
\]
An ideal $\ideal$ is called {\em Fr\'echet} if $\ideal=\ideal^{\perp\perp}$. We include  some examples  in section \ref{examples}.

Let $\{K_n:\; n\in F\}$ be a partition of a countable set $X$, where $F\subseteq \N$.
For $n\in F$, let $\ideal_n$ be an ideal on $K_n$. The direct sum, denoted by $\bigoplus\limits_{n\in F} \ideal_n$, is defined  by
\[
A\in \bigoplus_{n\in F} \ideal_n \Leftrightarrow (\forall n\in F)(A\cap K_n \in \ideal_n).
\]
In general, given a sequence of ideals $\ideal_n$ over a countable set $X_n$, we define $\bigoplus_n \ideal_n$ by taking a partition $\{K_n:\;
n\in \N\}$ of $\N$ and an isomorphic  copy $\ideal_n'$ of $\ideal_n$ on
$K_n$ and let $\bigoplus_n \ideal_n$ be $\bigoplus_n \ideal_n'$. It should be
clear that  $\bigoplus_n \ideal_n$ is,  up to isomorphism, independent of
the partition and the copies used. If all $\ideal_n$ are equal (isomorphic) to $\ideal$ we
will write $\ideal^\omega$ instead of $\bigoplus_n \ideal_n$.

\subsection{The space  $c_{0,\ideal}(X)$}

We begin by introducing the notion of $\mathcal I$-convergence for an ideal $\mathcal I$ on $\mathbb N$. This concept is due to Kostyrko, \v{S}al\'at and Wilczy\'nski \cite{k-s-w}.

\begin{definition}
Let $\mathcal I$ be an ideal on $\mathbb N$ and $X$ be a Banach space. A sequence $(x_n)$ in $X$ is $\mathcal I$-convergent to $x\in X$, and we write $\mathcal I-\lim x_n=x$, if for each $\varepsilon>0$, $\{n\in \N\,\colon\,\|x_n-x\|\geq \varepsilon\}\in \mathcal I$.
\end{definition}

The main objetive of this paper is to study the following subspace of $\ell_\infty$.
\begin{gather*}
c_{0,\mathcal I}(X)=\{(x_n)\in\ell_\infty(X)\,\colon\,\mathcal I-\lim x_n={\bf0}\}.
\end{gather*}

The upcomming result gives basic properties of $\mathcal I$-convergence
(see also \cite[Theorem 2.3]{k-m-s}). If ${\bf x}\in\ell_\infty(X)$ and $\varepsilon>0$, we will use through out the whole paper the following notation:
\begin{gather*}
    A(\varepsilon,{\bf x})=\{n\in\mathbb N\,\colon\,\|x_n\|\geq\varepsilon\}.
\end{gather*}

\begin{proposition}
\label{properties of A-convergence}
Let $X$ be a Banach space and $\mathcal I$ be an ideal on $\mathbb N$.
\begin{enumerate}
\item If $c\in\mathbb R$ and $\mathcal I-\lim x_n=x$, then $\mathcal I-\lim c x_n=cx$.

\item If $\mathcal I-\lim x_n=x$ and $\mathcal I-\lim y_n=y$, then $\mathcal I-\lim x_n+y_n=x+y$.

\item If $\mathcal I-\lim x_n={\bf0}$ and $\|y_n\|\leq\|x_n\|$ for all $n\in\mathbb N$, then $\mathcal I-\lim y_n={\bf0}$.
    
\item $c_{0,\mathcal I}(X)$ is closed subspace of $\ell_\infty(X)$.

\item If $X$ is a Banach lattice, then $c_{0,\mathcal I}(X)$ is a closed ideal of $\ell_\infty(X)$.
\end{enumerate} 
\end{proposition}

\begin{proof}
(1), (2) and (3) follow immediately from definition. (4):  If ${\bf y}\in\overline{c_{0,\mathcal I}(X)}^{\|\cdot\|_\infty}$ and $\varepsilon>0$ is given, there is ${\bf x}\in c_{0,\mathcal I}(X)$ such that $\|{\bf y-x}\|<\varepsilon/2$. Since $A(\varepsilon,{\bf y})\subset A(\varepsilon/2,{\bf x})$, we conclude that ${\bf y}\in c_{0,\mathcal I}(X)$. Finally for (5), if ${\bf y}\in c_{0,\mathcal I}(X)$ and $|{\bf x}|\leq|{\bf y}|$, then $\|x_n\|\leq\|y_n\|$ for each $n\in\mathbb N$. By (3), ${\bf x}\in c_{0,\mathcal I}(X)$.
\end{proof}




    When $X=\mathbb R$ we write $c_{0,\mathcal I}$ instead of $c_{0,\mathcal I}(X)$. The next result provides another description of $c_{0,\mathcal I}$. Recall that if ${\bf y}=(y_n)$ is a sequence, $\mathrm{supp}({\bf y})$ denotes the support of ${\bf y}$, that is, the set $\{n\in\mathbb N\,\colon\,y_n\neq0\}$.

\begin{proposition}
\label{span of I}
Let $X$ be a Banach space and  $\ideal$ a proper ideal on $\N$. 
\begin{enumerate}
\item $c_{0,\mathcal I}(X)=\overline{\{{\bf y}\in\ell_\infty(X)\,\colon\,\mathrm{supp}({\bf y})\in\mathcal I\}}^{\|\cdot\|_\infty}$. 

\item  $ c_{0,\ideal}=\overline{\mathrm{span}\{\chi_A\,\colon\,A\in\mathcal I\}}^{\|\cdot\|_\infty}$. 
 
\item  $A\in \ideal$ iff $\chi_A\in c_{0,\ideal}$. 
\end{enumerate}
\end{proposition}


\begin{proof}
(1) We first show  that if ${\bf y}=(y_n)\in\ell_\infty(X)$
and $\mathrm{supp}({\bf y})\in\mathcal I$, then ${\bf y}\in c_{0,\mathcal I}(X)$. In fact, let $\varepsilon>0$ be given. Clearly
$A(\varepsilon,{\bf y})\subset\mathrm{supp}({\bf y})$, and therefore $A(\varepsilon,{\bf} y)\in\mathcal I$. Since $c_{0,\ideal}(X)$ is closed in $\ell_\infty(X)$, then $\supseteq$ holds in the equation above. Conversely, let ${\bf x}=(x_n)\in c_{0,\ideal}(X)$. Fix $\varepsilon>0$ and let $A=A(\varepsilon,{\bf x})\in \ideal$. Pick a finite sequence of real numbers $\varepsilon=\lambda_1<\ldots<\lambda_k=\|(x_n)\|_\infty+1$ such that $\lambda_{i+1}-\lambda_{i}\leq \varepsilon$. Let $A_i=\{n\in A\,\colon\, \lambda_i\leq\|x_n\|< \lambda_{i+1}\}$ and ${\bf y}^i=(y_n^i)$, where
\begin{gather*}
    y_n^i=
    \begin{cases}
        \lambda_ix_n/\|x_n\|, & n\in A_i;\\
        0, &\mbox{otherwise,}
    \end{cases}
\end{gather*}
  for $1\leq i<k$. Notice that $A_i= 
  \mathrm{supp}({\bf y}^i)\in\mathcal I$ for each $1\leq i< k$
and $\|{\bf x}-({\bf y}^1+\ldots+{\bf y}^{k-1})\|_\infty\leq \varepsilon$. This shows $\subseteq$. 

(2) Suppose that $X=\mathbb R$. Now we let
\begin{gather*}
    A_i^+=\{n\in A_i\,\colon\,x_n>0\},\quad  A_i^-=\{n\in A_i\,\colon\,x_n<0\}, 
\end{gather*}
 ${\bf y}_1^i=\lambda_i\chi_{A_i^+}$ and ${\bf y}_2^i=-\lambda_i\chi_{A_i^-}$ for $1\leq i<k$. We have
$\|{{\bf x}-({\bf y}_1^1+\cdots+ {\bf y}_1^{k-1}+ {\bf y}_2^1+\cdots+{\bf y}_2^{k-1})}\|\leq\varepsilon$.

(3) The {\em only if part} is obvious. Suppose $\chi_A\in c_{0,\ideal}$, then $A=A(1/2, {\bf x})\in\mathcal I$ and we are done. 
\end{proof}

%
%
%

We recall that an ideal $\ideal$ over $\N$ can be identified (via characteristic functions) with a subset of the Cantor space $\{0,1\}^\N$, so that we can talk about Borel, meager ideals, etc.  It is  worth keeping  in mind the following result when  dealing  with meager ideals. 

\begin{theorem}
\label{Teorema de leonetti}
\cite{Leonetti2018} 
Let $\ideal$ be an ideal on $\N$. If $\ideal$ is meager, then $c_{0,\ideal}$ is not complemented in $\ell_\infty$, and, thus, not isomorphic to $\ell_\infty$. 
\end{theorem}

\section{Isomorphisms between  $c_{0,\ideal}$ spaces and  the Katetov order on ideals}

The classical Banach-Stone theorem states that given two locally compact Hausdorff spaces $K$ and $S$, $C_0(K)$ and $C_0(S)$ are isometric iff $K$ and $S$ are homeomorphic \cite[Theorem 4.1.5]{AK}. Holszty\'nski \cite{holsz} extends this result by showing that if $C_0(K)$ embeds isometrically in $C_0(S)$, then $K$ is a continuous image of a subspace of $S$. Our goal in this section is to prove versions of these results in the setting of $c_{0,\mathcal I}$ spaces.  We first look at isometry and later to  Banach lattice isometries. We first  recall some classical notions. 

Recall that $\beta\N$ is the Stone-Cech compactification of $\N$ which  is usually identified with the collection of all ultrafilters on $\N$. For a set $A\subset\mathbb N$, we let $A^*=\{p\in\beta\mathbb N\,\colon\,A\in p\}$.  The family $\{A^*\,\colon\,A\subset\mathbb N\}$ defines a basis for the topology of $\beta\mathbb N$. As usual, we identify each $n\in \N$ with the principal ultrafilter $\{A\subseteq\N: n\in A\}$.  Every principal ultrafilter is an isolated point of $\beta\N$. We set $U_\mathcal I=\{p\in\beta\mathbb N\,\colon\,\ideal\cap  p\neq \emptyset\}$.
Since $\displaystyle U_{\mathcal I}=\bigcup\{A^*\,\colon\,A\in\mathcal I\}$,  $U_{\mathcal I}$ is open in $\beta\mathbb N$.   Under the previous identification, $\N\subseteq U_\ideal$ for every ideal $\ideal$. 

Given a bounded sequence $(x_n)$ of reals numbers and an ultrafilter $p$ on $\N$, we denote by $p-\lim x_n$ the only $x\in\mathbb R$ such that $\{n\in \N: |x_n-x|<\varepsilon\}\in p$ for all $\varepsilon>0$.  Notice that if $p$ is the  principal ultrafilter $m$, then $p-\lim_nx_n=x_m$.

 It is well known that $c_{0,\mathcal I}$ and $C_0(U_\mathcal I)$ are isometric (see for instance \cite[p. 255]{gonzalez-kania} and \cite[p. 12]{kania}), nevertheless we include here a proof written  in terms  of $p$-limits which will be needed in the sequel. 

\begin{proposition}
\label{identification c_0,I and C_0(U_I)}
Let $\mathcal I$ be an ideal on $\mathbb N$. The map $\Phi_\mathcal I\colon c_{0,\mathcal I} \to C_0(U_\mathcal I)$ given by 
\begin{align*}
\Phi_{\mathcal I}({\bf x})(p)=p-\lim x_n,
\end{align*}
where ${\bf x}=(x_n)$ and   $p\in U_{\mathcal I}$, is a Banach lattice isometry. Moreover, $A\in\mathcal I$ iff $\chi_{A^*}\in C_0(U_\mathcal I)$ and $\Phi_{\mathcal I}(\chi_A)=\chi_{A^*}$ for each $A\in\mathcal I$.
\end{proposition}

\begin{proof}
Let  ${\bf x}=(x_n)\in c_{0,\mathcal I}$ be given. 
We  prove first  that $\Phi_{\mathcal I}({\bf x})\in C_0(U_{\mathcal I})$. Fix $p\in U_\ideal$ and let $r>0$. To show  that $\Phi_{\mathcal I}({\bf x})$ is continuous at $p$, suppose that $|\Phi_{\mathcal I}({\bf x})(p)-a|=|(p-\lim x_n)-a|<r$. Let $\delta>0$ be such that $|(p-\lim x_n)-a|<r-\delta$ and  $A=\{n\in\mathbb N\,\colon\,|x_n-a|<r-\delta\}$. Note that $p\in A^*$ and for $q\in A^*$ we have $|(q-\lim x_n)-a|\leq r-\delta<r$. Hence, $p\in A^*\subset \Phi_{\mathcal I}^{-1}({\bf x})((a-r,a+r))$. Thus $\Phi_{\mathcal I}({\bf x})$ is continuous at $p$.
        
To see that $\Phi_{\mathcal I}({\bf x})$ vanishes at infinity, let $\varepsilon >0$ and consider $A_\varepsilon=\{n\in\mathbb N\,\colon\,|x_n|>\varepsilon/2\}$. If $p\in U_{\mathcal I}$ satisfies $|\Phi_{\mathcal I}({\bf x})(p)|=|p-\lim x_n|\geq\varepsilon$, then $A_\varepsilon\in p$,  that is $p\in A_\varepsilon^*$. Hence $\{p\in U_{\mathcal I}\,\colon\,|\Phi_{\mathcal I}({\bf x})(p)|\geq\varepsilon\}\subset A_\varepsilon^*$. Thus $\{p\in U_{\mathcal I}\,\colon\,|\Phi_{\mathcal I}({\bf x})(p)|\geq\varepsilon\}$ is compact for each $\varepsilon>0$.

Clearly $\Phi_{\mathcal I}$ is linear. Now we prove that $\Phi_{\mathcal I}$ is an isometry. Let ${\bf x}\in c_{0,\mathcal I}$ be given and $a=\|{\bf x}\|_\infty$. Note that $|\Phi_{\mathcal I}({\bf x})(p)|=|p-\lim x_n|\leq \|{\bf x}\|_\infty$. On the other hand, for $0<\delta<a$, we have that $A_\delta=\{n\in\mathbb N\,\colon\,|x_n|>a-\delta\}\in\mathcal I$. Take $p\in A_\delta^*
$. Whence, $p\in U_{\mathcal I}$ and $\Phi_{\mathcal I}({\bf x})(p)=p-\lim x_n>a-\delta/2$. Since $\delta>0$ was arbitrary,  $\|\Phi_{\mathcal I}({\bf x})\|\geq a$. 

To see that  $\Phi_{\mathcal I}$ is onto, fix $f\in C_0(U_{\mathcal I})$. Consider the sequence ${\bf x}=(x_n)$ given by $x_n=f(n)$ for $n\in \N$. For every $\varepsilon>0$ we have
\begin{gather*}
A(\varepsilon,{\bf x})\subset\{p\in U_{\mathcal I}\,\colon\,|f(p)|\geq\varepsilon\}\subset U_{\mathcal I}=\bigcup\{A^*\,\colon\,A\in\mathcal I\}.
\end{gather*}
By compactness, there are $A_1,\ldots,A_n\in\mathcal I$ such that $A(\varepsilon,{\bf x})\subset\bigcup_{j=1}^nA_j$.
Whence $A(\varepsilon,{\bf x})\in\mathcal I$ and thus ${\bf x}\in c_{0,\mathcal I}$. Since $f$ is continuous, we have
\begin{gather*}
    \Phi_{\mathcal I}({\bf x})(p)=p-\lim x_n=p-\lim f(n)=f(p),\quad\mbox{for all $p\in U_{\mathcal I}$.}
\end{gather*}
Hence, $\Phi_{\mathcal I}({\bf x})=f$. 

The identity $\Phi_{\mathcal I}(\chi_A)=\chi_{A^*}$ for each $A\in\mathcal I$ follows from the fact  that $\chi_{A^*}(p)=p-\lim\chi_A(n)$ for every $p\in \beta\mathbb N$ and all $A\subset\mathbb N$. Finally, it is easy to see that $\Phi_{\mathcal I}$ is a Banach lattice isomorphism.
\end{proof}

The proof of the following result is straightforward. 

\begin{proposition}
\label{Th}
Let $\mathcal I$ and $\mathcal J$ be ideals on $\mathbb N$ and  $h:\N\to\N$ be an isomorphism between $\ideal$ and $\idealj$. Let $T_h\colon c_{0,\mathcal I}\to c_{0,\mathcal J}$ be given by $T_h({\bf x})= {\bf x}\circ h$. Then $T_h$ is a Banach lattice isometry. 
\end{proposition}

\begin{theorem}
\label{c_0,I caracteriza ideales}
Let $\mathcal I$ and $\mathcal J$ be ideals on $\mathbb N$.  
Then $c_{0,\mathcal I}$ and $c_{0,\mathcal J}$ are isometric if, and only if,  $\mathcal I$ and $\mathcal J$ are isomorphic.
\end{theorem}

\begin{proof}
One direction is Proposition \ref{Th}. 
For the other, let  $T\colon c_{0,\mathcal I}\to c_{0,\mathcal J}$ be an isometric isomorphism. Then $\Phi_{\mathcal J}\circ T\circ\Phi_{\mathcal I}^{-1}\colon C_0(U_\mathcal I)\to C_0(U_{\mathcal J})$ is an isometry.
By the classical Banach-Stone theorem, there are a homeomorphism $f\colon U_{\mathcal J}\to U_{\mathcal I}$ and a continuous map $\sigma\colon U_{\mathcal J}\to\{-1,1\}$ such that
$\Phi_{\mathcal J}\circ T\circ\Phi_{\mathcal I}^{-1}(F)=\sigma F\circ f$ for all $F\in C_0(U_{\mathcal I})$. Thus 
$(\Phi_{\mathcal J}\circ T)({\bf x})=\sigma\Phi_\mathcal I({\bf x})\circ f$ for all ${\bf x}\in c_{0,\mathcal I}$. It follows from Proposition \ref{identification c_0,I and C_0(U_I)} that if $A\in\mathcal I$ is given, then
\begin{gather}
\label{equation 1}
(\Phi_{\mathcal J}\circ T)(\chi_A)(p)=\sigma(p)\Phi_\mathcal I(\chi_A)(f(p))=\sigma(p)\chi_{A^*}(f(p))=\sigma(p)\chi_{f^{-1}(A^*)}(p),\quad\mbox{for all $p\in U_{\mathcal J}$.}
\end{gather}
On the other hand, if $p\in U_{\mathcal J}$, we have
\begin{gather}
\label{equation 2}
\chi_{f^{-1}(A^*)}(p)=p-\lim\chi_{A^*}(f(n))=p-\lim\chi_{A}(f(n))=p-\lim\chi_{f^{-1}(A)}(n)=\chi_{(f^{-1}(A))^*}(p).
\end{gather}
By combining \eqref{equation 1} and \eqref{equation 2} we obtain that 
$(\Phi_{\mathcal J}\circ T)(\chi_A)=\sigma\chi_{(f^{-1}(A))^*}$ for all $A\in\mathcal I$. Hence, 
\begin{gather*}
    T\chi_A=s\chi_{f^{-1}(A)}\quad\mbox{for all $A\in\mathcal I$,}
\end{gather*}
where $s=\sigma\restriction\mathbb N$. So $h=f\restriction\mathbb N$ is a bijection from $\mathbb N$ to $\mathbb N$ such that $A\in\mathcal I$ iff $h^{-1}(A)\in\mathcal J$. Thus $\ideal$ and $\idealj$ are isomorphic.
\end{proof}

\begin{theorem}
\label{ab2}
Let $\ideal$ and $\idealj$ be ideals on $\N$. Then, $\ideal\sqsubseteq\idealj$
if, and only if, there is  an into Banach isometry $T\colon c_{0,\mathcal I}\to c_{0,\mathcal J}$ such that $T[c_{0,\ideal}]$ is an ideal. 
\end{theorem}

\proof  Suppose $\ideal\sqsubseteq \idealj$ and let $h:\N\to\N$ be a bijective Kat\v{e}tov reduction. Let $T_h\colon c_{0,\mathcal I}\to c_{0,\mathcal J}$ given by  $T_h({\bf x})= {\bf x}\circ h$.  By Proposition \ref{Th},  $T_h$ is the required into Banach isometry (which is moreover a lattice isometry). 

Conversely, suppose such $T$ exists. By Proposition \ref{ideal2}, there is an ideal $\idealj'$ such that  $T[c_{0,\ideal}]=c_{0,\idealj'}$. By Theorem \ref{c_0,I caracteriza ideales}, $\ideal$ is isomorphic to $\idealj'$. Clearly $\idealj'\subseteq \idealj$, thus $\ideal\sqsubseteq \idealj$.
\endproof

Next we will look at a condition stronger than  isometries, namely, we work with  Banach lattice isometries. This will be related to the  Kat\v{e}tov pre-order on ideals.

Let $\mathcal I$ and $\mathcal J$ be ideals on $\mathbb N$ such that $\ideal\leq_K\idealj$, that is, there is  $h\colon\mathbb N\to\mathbb N$ such that  $h^{-1}(A)\in\mathcal J$ whenever $A\in\mathcal I$. We can always assume $h$ to be onto (see Proposition \ref{katetoveOnto}).  Recall the map $T_h$ as in Proposition \ref{Th}.

\begin{proposition}
\label{aaa}
Let $\mathcal I$ and $\mathcal J$ be ideals on $\mathbb N$ with $\ideal\leq_K\idealj$,  $h:\N\to \N$ be an onto Kat\v{e}tov reduction from $\ideal$ to $\idealj$ and $T_h$ as above.  Then,

\begin{enumerate}
\item $T_h$ is an into Banach lattice isometry. Moreover, it satisfies the following: 

\begin{enumerate}

\item For each $n\in\mathbb N$, there is $A\in\mathcal I$ satisfying $T\chi_A(n)=1$. 

\item For all $\emptyset \neq \mathcal{F}\subseteq \ideal$  such that $\bigcap_{A\in \mathcal{F}}A=\emptyset$, we have $\bigwedge_{A\in \mathcal{F}}T(\chi_A)=0$.
\end{enumerate}

\item  $h\colon\mathbb N\to\mathbb N$ is bijective iff $T_h[c_{0,\ideal}]$ is an ideal of $c_{0,\idealj}$. 
    
\end{enumerate}
\end{proposition}

\proof (1) is  straightforward. 

(2)  Suppose $h$ is bijective and let ${\bf x}=(x_n)\in c_{0,\ideal}$ and ${\bf y}=(y_n)\in c_{0,\idealj}$ be such that $|{\bf y}|\leq |T_h({\bf x})|=T_h(|{\bf x}|)$. For each $n\in\mathbb N$, let 
$z_n=y_{h^{-1}(n)}$ and ${\bf z}=(z_n)$. Then $|z_n|\leq |x_n|$ for all  $n\in\mathbb N$, thus ${\bf z}\in c_{0,\ideal}$ and clearly $T_h({\bf z})={\bf y}$. Conversely, suppose that $T_h[c_{0,\ideal}]$ is an ideal and that there are $n\neq m$ such that $h(n)=h(m)=k$. Then $T_h(\chi_{\{k\}})=\chi_{h^{-1}\{k\}}\geq \chi_{\{n\}}$. Thus, there is ${\bf x}\in c_{0,\ideal}$ such that $T_h({\bf x})=\chi_{\{n\}}$. In particular $0=T_h({\bf x})(m)=x_{h(m)}=x_{h(n)}=T_h({\bf x})(n)=1$, a contradiction.
\endproof

Now we proceed to prove the converse of Proposition \ref{aaa}. 

\begin{theorem}
\label{ab}
Let $\ideal$ and $\idealj$ be ideals on $\N$. 
Let $T\colon c_{0,\mathcal I}\to c_{0,\mathcal J}$ be an into Banach lattice  isometry such that $T[c_{0,\ideal}]$ is an ideal containing $c_0$. Then 
$T=T_h$ for some  bijection $h:\N\to\N$. 


\end{theorem}

\proof
First, we claim that for every $A\in\ideal$ there is $B\in \idealj$ such that $T(\chi_A)=\chi_B$.
In fact, we show that for all $m\in \N$ there is $k\in \N$ such that $T(\chi_{\{m\}})=\chi_{\{k\}}$. Let $k\in \N$ and $a>0$ be such that
$a\chi_{\{k\}}\leq T(\chi_{\{m\}})$.
Since $T[c_{0,\ideal}]$ is an ideal, there is ${\bf y}\in c_{0,\ideal}$ such that $T({\bf y})=a\chi_{\{k\}}$. Clearly $0< {\bf y}\leq \chi_{\{m\}}$ and $\|{\bf y}\|=a$. Therefore ${\bf y}=a\chi_{\{m\}}$.  Thus $T(\chi_{\{m\}})=\chi_{\{k\}}$. 

Now let $A\in \ideal$ and put $B=\{n\in\N: T(\chi_A)(n)=1\}$. We claim  that $\chi_B= T(\chi_A)$.   Clearly $\chi_B\leq T(\chi_A)$. Since $T[c_{0,\ideal}]$ is an ideal, there is ${\bf x}\in c_{0,\mathcal I}$ such that $T({\bf x})=\chi_B$.  For each $k\in A$, let $m_k\in\mathbb N$ be such that $T(\chi_{\{k\}})=\chi_{\{m_k\}}$. Notice that $m_k\in B$ since 
$T(\chi_{\{k\}})\leq T(\chi_A)$  for all $k\in A$. Thus
$T(\chi_{\{k\}})=\chi_{\{m_k\}}\leq \chi_B= T({\bf x})$. So, 
$\chi_{\{k\}}\leq {\bf x}$ for all $k\in A$. Therefore $\chi_A\leq {\bf x}$ and hence $T(\chi_A)\leq T({\bf x})=\chi_B$ and we are done.

Next we show show that for all $n\in \N$, there is $m\in \N$ such that $T(\chi_{\{m\}})=\chi_{\{n\}}$.
In fact, let $n\in\N$. Since $\chi_{\{n\}}\in  c_{0}$, there is  ${\bf x}\in c_{0,\ideal}$ with ${\bf x}>{\bf0}$ and $T({\bf x})=\chi_{\{n\}}$. Let $m\in\mathbb N$ be such that $x_m>0$. Since $x_m\chi_{\{m\}}\leq {\bf x}$, we have $x_m T(\chi_{\{m\}})\leq \chi_{\{n\}}$. So, $m=n$ and  ${\bf x}=\chi_{\{m\}}$.

\medskip 

From the first claim, there is $f\colon\N\to\N$ such that $T(\chi_{\{m\}})=\chi_{\{f(m)\}}$. Since $T$ is an isometry, $f$ is injective. Since $c_0\subseteq T[c_{0,\ideal}]$, by (2) $f$ is onto. Let $h=f^{-1}$. From the proof of (1) it follows that $T(\chi_A)=\chi_{f(A)}$ for every $A\in \ideal$. This shows that $T=T_h$.
\endproof

%

\begin{theorem}
\label{HT}
Let $\mathcal I$ and $\mathcal J$ be ideals on $\mathbb N$ and  $T\colon c_{0,\mathcal I}\to c_{0,\mathcal J}$  be  an into Banach lattice isometry.
Suppose $T$ has the following properties:
    
\begin{enumerate}
\item for each $n\in\mathbb N$, there is $A\in\mathcal I$ satisfying $T\chi_A(n)=1$. 

\item  For all $\emptyset \neq \mathcal{F}\subseteq \ideal$  such that $\bigcap_{A\in \mathcal{F}}A=\emptyset$, we have $\bigwedge_{A\in \mathcal{F}}T(\chi_A)=0$.
\end{enumerate}

\noindent Then  $\mathcal I\leq_{K}\mathcal J$.



\end{theorem}

\begin{proof}
Let $T\colon c_{0,\mathcal I}\to c_{0,\mathcal J}$ be an into Banach lattice isometry and define $\widetilde T\colon C_0(U_{\mathcal I})\to C_0(U_{\mathcal J})$ by $\widetilde T=\Phi_{\mathcal J}\circ T\circ\Phi_{\mathcal I}^{-1}$. Notice that $\widetilde T$ is a Banach lattice isometry and for each $n\in\mathbb N$, there is $A\in\mathcal I$ such that $\widetilde T\chi_{A^*}(n)=1$. Following a proof of Holszty\'nski theorem (as presented in \cite{rincon-villamizar})  we define,  for $p\in U_{\mathcal I}$,
\[
\begin{array}{rcl}
    \Delta(p) & = &\{q\in U_{\mathcal J}\,\colon\,\widetilde T^*\delta_{q}(\{p\})=1\},\quad\mbox{and}\\
    \\
\Delta &= & \displaystyle\bigcup_{p\in U_{\mathcal I}}\Delta(p).
\end{array}
\]
The following is a crucial ingredient, its proof can be found in  \cite{rincon-villamizar}.

\begin{claim}
\label{deltap}
The following statements are valid:
    
\begin{enumerate}
\item  $\Delta(p)\neq\emptyset$ for all $p\in U_{\mathcal I}$.
 
\item  $\Delta(p)\cap\Delta(p')=\emptyset$ if $p\neq p'$.
        
\item The map $\phi\colon\Delta\to U_{\mathcal I}$ defined by $\phi(q)=p$ iff $q\in\Delta(p)$ is continuous and onto. 
        
\item $\widetilde TF(q)=\sigma(q)F(\phi(q))$ for all $F\in C_0(U_{\mathcal I})$ and $q\in\Delta$, where $\sigma\colon\Delta\to\mathbb R$ is continuous and $|\sigma(q)|=1$ for all $q\in\Delta$. 
\end{enumerate}

\end{claim}

Since $\widetilde T$ is a Banach lattice isometry, we have $\sigma(q)=1$ for each $q\in\Delta$.

\begin{claim}\label{primer claim}
For each $p\in U_{\mathcal I}$ we have
\begin{gather*}
\Delta(p)=\{q\in U_{\mathcal J}\,\colon\,\widetilde TF(q)=1, \mbox{whenever $F\geq{\bf0}$ and $F(p)=1=\|F\|$}\}.
\end{gather*}

\end{claim}

Proof of Claim \ref{primer claim}: Suppose that $\widetilde T\delta_q(\{p\})=1$. Let $F\in C_0(U_{\mathcal I})$ be such that $F\geq{\bf0}$ and $F(p)=1=\|F\|$. Since $\chi_{\{p\}}\leq F$, we have $\widetilde TF(q)=\widetilde T^*\delta_q(F)\geq\widetilde T^*\delta_q(\chi_{\{p\}})=1$. 

For the other inclusion, let $q\in U_{\mathcal J}$ be such that $\widetilde TF(q)=1$ whenever $F\geq{\bf0}$ and $F(p)=1=\|F\|$. Also let $\mathcal V_p$ be a fundamental system of open neighborhoods of $p$. If $\varepsilon>0$ is given, there is an open $V\subset U_{\mathcal I}$ with $p\in V$ and $\widetilde T^*\delta_q(V\setminus\{p\})<\varepsilon$. For each $W\in\mathcal V_p$, let $F_W\in C_0(U_{\mathcal I})$ be such that $F_W(p)=1=\|F_W\|$, $F_W\geq{\bf0}$ and $F_W(U_{\mathcal I}\setminus W)=\{0\}$. Take $W_0\in\mathcal V_p$ with $W_0\subset V$. So, 
\begin{gather*}
    1=\widetilde TF_{W_0}(q)=\int_{U_{\mathcal I}}F_{W_0}\,d\widetilde T^*\delta_q=\int_{W_0}F_{W_0}\,d\widetilde T^*\delta_q\leq\widetilde T^*\delta_q(\{p\})+\widetilde T^*\delta_q(V\setminus\{p\})\leq1+\varepsilon.
\end{gather*}
Since $\varepsilon>0$ is arbitrary, we conclude that $\widetilde T^*\delta_q(\{p\})=1$, that is, $q\in\Delta(p)$. Thus we have proved Claim \ref{primer claim}.

\begin{claim}\label{segundo claim}
For each $p\in U_{\mathcal I}$ we have
\begin{gather*}
    \Delta(p)=\{q\in U_{\mathcal J}\,\colon\,\mbox{$\widetilde T\chi_{A^*}(q)=1$ for all $A\in\mathcal I$ with $p\in A^*$}\}.
\end{gather*}

\end{claim}

Proof of Claim \ref{segundo claim}: The inclusion $\subseteq$ is clear from Claim \ref{primer claim}. For the other direction, let $\widetilde\Delta(p)$ be the the right hand side of the identity above and $q\in \widetilde\Delta(p)$. Let $\varepsilon>0$ be given. If $F\in C_0(U_{\mathcal I})$ satisfies $F\geq{\bf0}$ and $F(p)=1=\|F\|$, then $B=\{n\in\mathbb N\,\colon\,|F(n)-1|<\varepsilon\}\in p$. Let $A\in\mathcal I\cap p$. Since $q\in\widetilde\Delta(p)$, $\widetilde T\chi_{(A\cap B)^*}(q)=1$. As  $(1-\varepsilon)\chi_{(A\cap B)^*}\leq F$ and $T$ is lattice order preserving,  we have that $1-\varepsilon\leq TF(q)$. Since $\varepsilon$ was arbitrary,   $TF(q)=1$ and $q\in \Delta(p)$. Thus we have proved Claim \ref{segundo claim}.
   
\begin{claim}\label{tercer claim}
 $\mathbb N\subset\Delta$.
\end{claim}

Proof of Claim \ref{tercer claim}: Fix $n\in\N$ and let $\mathcal F_n=\{A\in\mathcal I\,\colon\, T\chi_{A}(n)=1\}$.   Thus, from the hypothesis (1), $\mathcal F_n$ is non-empty. 
Now if $A,B\in\mathcal F_n$, then
$T\chi_{A}(n)=1$ and $T\chi_{B}(n)=1$. Since $ T\chi_{(A\cap B)}= T(\chi_{A}\wedge\chi_{B})= T(\chi_{A})\wedge T(\chi_{B})$, we conclude that $T\chi_{(A\cap B)}(n)=1$. Hence $\mathcal F_n$ has the finite intersection property. Let $p\in\beta\mathbb N$ be such that $\mathcal F_n\subset p$.  We will show that $n\in\Delta(p)$. 

 Notice that $\widetilde T\chi_{A^*}(q)=q-\lim_m T\chi_{A}(m)$, for every $q\in U_\ideal$ and $A\in \ideal$. In particular,  when $q$ is the principal ultrafilter $n$, we get $\widetilde T\chi_{A^*}(n)=T\chi_A(n)$. Thus $\mathcal F_n=\{A\in\mathcal I\,\colon\, \widetilde T\chi_{A^*}(n)=1\}$.

As $\mathcal{F}_n\subseteq p$, we have that $p\in\mathcal U_{\mathcal I}$. Let $A\in\mathcal F_n$ be fixed. To show that $n\in \Delta(p)$ it suffices to have, by Claim \ref{segundo claim}, that $\widetilde T\chi_{B^*}(n)= 1$ for all  $B\in\mathcal I$ with $p\in B^*$. Fix such $B$ and suppose,  towards a contradiction,  that $\widetilde T\chi_{B^*}(n)\neq 1$. From  Claim \ref{deltap} (4), we have $\widetilde T\chi_{B^*}(n)=\chi_{B^*}(\phi(n))$. Thus  $\widetilde T\chi_{B^*}(n)=0$. On the other hand,  
    \begin{gather*}
        1=T\chi_{A}(n)= T\chi_{(A\cap B)}(n)+T\chi_{(A\setminus B)}(n).
    \end{gather*}
As  $0\leq T\chi_{A\cap B}(n)\leq T\chi_{B}(n)$,  we conclude that $T\chi_{(A\setminus B)}(n)=1$ and $A\setminus B\in p$, which is a contradiction. Thus we have proved Claim \ref{tercer claim}.

Now we show that $\mathbb N\subseteq \phi^{-1}(\mathbb N)$  where $\phi$ is as in Claim \ref{deltap}. 
Let  $n\in \N$. Then $\mathcal F_n=\{A\in\mathcal I\,\colon\, T\chi_{A}(n)=1\}$. We claim  that  $\bigcap\{A:\; A\in \mathcal F_n\}\neq \emptyset$. In fact, we have $ \chi_{\{n\}}\leq T\chi_A$
for all $A\in \mathcal{F}_n$. Thus 
\[
0<\bigwedge_{A\in \mathcal{F}_n} T\chi_A.
\]
Thus, by the hypothesis (2), there is $m$  such that  $m\in A$ for all $A\in \mathcal F_n$. Therefore $n\in \Delta(m)$, that is, $\phi(n)=m$. 

To end the proof, we let $h:=\phi\restriction \N$. By Claim \ref{deltap}(3), $h$ is onto. We show that  $h: \N\to\mathbb N$ is a Kat\v{e}tov reduction. We have
\begin{gather}
\label{4.10}
    T\chi_A(n)=\widetilde T\chi_{A^*}(n)=\chi_{A^*}(\phi(n))=\chi_A(\phi(n))=\chi_{\phi^{-1}(A)}(n)=\chi_{h^{-1}(A)}(n)
\end{gather}
for all $A\in\mathcal I$ and $n\in \N$.
Thus, if $A\in\mathcal I$, then $\chi_{h^{-1}(A)}\in c_{0,\mathcal J}$ which means that $h^{-1}(A)\in\mathcal J$. This shows that $\ideal\leq_K\idealj$ and we are done. 
\end{proof}

%
%

We end this section commenting about the role of the condition (1) in Theorem \ref{HT}.  We show first  a simple way to construct subspaces of $c_{0,\mathcal I}$, a particular case  was already observed in \cite{Leonetti2018}.

\begin{proposition}
\label{restriction}
Let $\mathcal I$ be an ideal on $\mathbb N$ and $A\subseteq \N$. Then $c_{0,\ideal\restriction A}$ is Banach lattice isometric to a closed ideal of $c_{0,\ideal}$. 
\end{proposition}

\begin{proof} 
Recall that $\ideal\restriction A$ is an ideal on $A$ and thus $c_{0,\ideal\restriction A}$ consists of sequences on $A$. For each ${\bf x}=(x_n)_{n\in A}\in c_{0,\ideal\restriction A}$,  let $T({\bf x}):=(\chi_A(n)x_n)_{n\in \N}$. Clearly $T({\bf x})\in c_{0,\ideal}$.  The map $\Psi$ is the required Banach lattice isometry. It is not difficult to check that $T(c_{0,\mathcal I\restriction A})$ is an ideal of $c_{0,\mathcal I}$.
\end{proof}

In the previous result suppose $\ideal$ is not $\mathcal{P}(\N)$  and $A\in\ideal$ is infinite, then $\ideal\restriction A=\mathcal{P}(A)$ and thus $\ideal\restriction A\not\leq_K \ideal$. In this case condition (1) in Theorem \ref{HT} fails because  $T(\chi_B)(n)=0$ for all $B$ and any $n\not\in A$.

\section{Some lattices properties of $c_{0,\ideal}$}
\label{latticesProp}

In this section we present some properties of $c_{0,\ideal}$ as an ideal of the  Banach lattice $\ell_\infty$.
First,  we characterize the closed ideals of $\ell_\infty$. Part (iii)  below says that it is somewhat harmless to assumed that every closed ideal of $\ell_\infty$ contains $c_0$. 
The  particular case of (i)  below for a maximal ideal was shown in \cite[Theorem 4.1]{k-m-s}.

\begin{proposition}
\label{ideal2}
Let $Y$ be a closed sublattice of $\ell_\infty$. Then

\begin{itemize}
\item[(i)]$Y$ is an ideal iff there is an ideal $\ideal$ (not containing necesarely $\fin$) on $\N$ such that $Y=c_{0,\ideal}$.

\item[(ii)] $c_0\subseteq c_{0, \ideal}$ iff $\fin\subseteq \ideal$.

\item[(iii)] For each $A\subseteq \N$, let $Z_A=\{{\bf x}\in \ell_\infty: x_n=0\;\mbox{for all $n\in A$}\}$.   For each closed ideal $Y$ of $\ell_\infty$, let $A=\{n\in \N: \chi_{\{n\}}\not\in Y\}$, then  $Y=\langle Y\cup c_0\rangle\cap Z_A$.

\end{itemize} 

\end{proposition}

\proof
(i) We have already seen that every $c_{0,\mathcal I}$ is an ideal of $\ell_\infty$. Conversely, if $Y$ is an ideal of $\ell_\infty$, let $\ideal=\{A\subseteq \N: \; \chi_A\in Y\}$. It is easy to verify that $\ideal$ is an ideal. From  Proposition \ref{span of I}, it follows that $c_{0,\ideal}\subseteq Y$. For the other inclusion, fix ${\bf x}\in Y$ and $\varepsilon>0$.   By an argument analogous to that used in the proof of Proposition \ref{span of I}(2), there is ${\bf y}\in \mathrm{span}\{\chi_A\,\colon\,A\in\mathcal I\}$ such that $\|{\bf x}-{\bf y} \|<\varepsilon$. 

(ii) is straightforward. (iii) Firstly, we prove that $Y\subseteq Z_A$. Assume that $Y\setminus Z_A\neq\emptyset$ and let ${\bf z}=(z_n)\in Y\setminus Z_A$. So, there is $m\in A$ such that $z_m\neq0$. Since $|z_m|\chi_{\{m\}}\leq|{\bf z}|$ and $Y$ is an ideal, $z_m\chi_{\{m\}}\in Y$. Hence, $\chi_{\{m\}}\in Y$ which is impossible. On the other hand, it is clear that  $Y\subseteq\langle Y\cup c_0\rangle$. For the other inclusion, let $\mathcal I$ be the ideal given by (i), i.e, such that $c_{0,\mathcal I}=Y$. By Theorem \ref{union of ideals} below we have
\begin{gather*}
    \langle Y\cup c_0\rangle=\langle Y+c_0\rangle=\langle c_{0,\mathcal I}+c_{0,\fin}\rangle=\langle c_{0,\mathcal I\sqcup\fin}\rangle=c_{0,\mathcal I\sqcup\fin}.
\end{gather*}
Now, if ${\bf x}\in\langle Y\cup c_0\rangle\cap Z_A$, then ${\bf x}\in c_{0,\mathcal I\sqcup\fin}$. Let $\varepsilon>0$ be given. Observe that $A({\bf x},\varepsilon)\cap A=\emptyset$. If $C\in\fin$ and $B\in\mathcal I$ satisfy $A({\bf x},\varepsilon)=B\cup C$, we have $C\in\mathcal I$ since $C\cap A=\emptyset$. So, $A({\bf x},\varepsilon)\in\mathcal I$. Whence, ${\bf x}\in c_{0,\mathcal I}=Y$.
\endproof

\begin{definition}
 \begin{enumerate}

\item  Two elements $x,y\in\ell_\infty$ are called \textit{$c_0$-disjoint}, and we write $x\perp_{c_0} y$, if $|x|\wedge|y|\in c_0$.

\item Two subspaces $S$ and $T$ of $\ell_\infty$ are called {\em $c_0$-disjoint} if $s\perp_{c_0} t$ for each $s\in S$ and $t\in T$.

\item The {\em $c_0$-disjoint complement} of a subspace $Y$ of $\ell_\infty$, denoted by $c_0[Y]$, is
\[
c_0[Y]=\{x\in \ell_\infty\,\colon\,\mbox{$x\perp_{c_0} y$ for all $y\in Y$}\}=\{x\in \ell_\infty\colon\;\mbox{$|x|\wedge|y|\in c_0$ for all $y\in Y$}\}.
\]
\end{enumerate} 
\end{definition}

\begin{theorem}
Let $\ideal$ and $\idealj$ be two ideals over $\N$. 
Then 

\begin{itemize}
\item[(i)] $c_{0,\ideal}$ and $c_{0,\idealj}$ are $c_0$-disjoint iff $\ideal\subseteq \idealj^\perp$.  In particular, $c_{0,\ideal}$ and $c_{0, \ideal^\perp}$ are $c_0$-disjoint for every ideal $\ideal$.

\item[(ii)] $c_{0,\ideal^\perp}= c_0[c_{0,\ideal}]$.

\item[(iii)] $\ideal$ is Fr\'echet iff $c_0[c_0[c_{0,\ideal}]]=c_{0,\ideal}$. 

\end{itemize}
\end{theorem}

\begin{proof}
\begin{itemize}
\item[(i)] Suppose $\ideal\subseteq \idealj^\perp$ and let ${\bf x}=(x_n)\in c_{0,\ideal}$ and ${\bf y}=(y_n)\in c_{0,\idealj}$.  Let $\varepsilon>0$. Then $A(\varepsilon,{\bf x})=\{n
\in \N: |x_n|\geq \varepsilon\}\in \ideal$. Thus $A(\varepsilon,{\bf x})\in \idealj^\perp$. On the other hand, $A(\varepsilon,{\bf y})=\{n \in \N: |y_n|\geq \varepsilon\}\in \idealj$. Thus 
$A(\varepsilon,{\bf x})\cap A(\varepsilon,{\bf y})$ is finite. Thus there is $n_0$ such that for all $n>n_0$, $|x_n|<\varepsilon$ or $|y_n|<\varepsilon$. Thus $\min\{|x_n|, |y_n|\}<\varepsilon$ for all $n>n_0$. Hence  $|{\bf x}|\wedge  |{\bf y}|\in c_0$. 

Conversely, suppose $\ideal\not\subseteq \idealj^\perp$. Let $A\in \ideal\setminus\idealj^\perp$. Then there is $B\subseteq A$ infinite such that $B\in \idealj\cap \ideal$. Hence $\chi_B\in c_{0,\ideal}\cap c_{0,\idealj}$ and clearly $\chi_B\not\in c_0$ as $B$ is infinite.

For the second claim, we recall that $\ideal\subseteq \ideal^{\perp\perp}$ for every ideal $\ideal$.

\item[(ii)]  The inclusion $\subseteq$ follows from (i). Conversely, suppose ${\bf x}=(x_n)\not\in c_{0,\ideal^\perp}$ and let $1>\varepsilon>0$ be such that $A(\varepsilon,{\bf x}) \not\in \ideal^\perp$. Thus there is $B\in \ideal$ such that $C:= A(\varepsilon,{\bf x}) \cap B$ is infinite.
Then $\chi_C\in c_{0,\ideal}$ and $\varepsilon\chi_C\leq |{\bf x}|\wedge \chi_C$. Thus $|{\bf x}|\wedge \chi_C\not\in c_0$. 

\item[(iii)] Suppose $\ideal$ is Fr\'echet. By (ii), we have 
\[
c_0[c_0[c_{0,\ideal}]]=c_0[c_{0,\ideal^\perp}]=c_{0,\ideal^{\perp\perp}}=c_{0,\ideal}.
\]
Conversely, to see that $\ideal$ is Fr\'echet it suffices to show that $\ideal^{\perp\perp}
\subseteq \ideal$. Let $B\in \ideal^{\perp\perp}$. By (ii) and our assumption, $\chi_B\in c_{0,\ideal^{\perp\perp}}=c_0[c_{0,\ideal^\perp}]=c_0[c_0[c_{0,\ideal}]]=c_{0,\ideal}$.  From Proposition \ref{span of I}, we conclude $B\in \ideal$.\qedhere
\end{itemize}
\end{proof}

Recall that an ideal $\mathcal I$ on $\mathbb N$ is \textit{tall} if every infinite subset of $\mathbb N$ contains a infinite member of $\mathcal I$.  A subset $A$ of a Banach lattice $E$ is called \textit{order dense in $B$} if for each ${\bf0}\neq x\in B$ there exists $a\in A$ such that ${\bf0}<|a|\leq|x|$. Observe that $c_{0, \ideal}$ is order dense in $\ell_\infty$ for any ideal $\ideal$ (containing $\fin$).

\begin{theorem}
Let $\mathcal I$ be an ideal on $\N$. The following statements are equivalent:
\begin{enumerate}
    \item $\mathcal I$ is tall;
    \item $c_{0,\mathcal I}\setminus c_0$ is order dense in $\ell_\infty\setminus c_{0}$;
    \item $c_{0,\ideal^{\perp}}=c_0$.
\end{enumerate}
\end{theorem}

\begin{proof}
(1) $\Rightarrow$ (2): Suppose that $\mathcal I$ is tall and let ${\bf0}\neq{\bf x}=(x_n)\in\ell_\infty\setminus c_{0}$ be given. For some $\varepsilon_0>0$ we have $A(\varepsilon_0,{\bf x})$ is infinite. The tallness of $\mathcal I$ implies that there exists $B\in\mathcal I$ infinite satisfying $B\subset A(\varepsilon_0,{\bf x})$. So, ${\bf y}=\chi_B\cdot{\bf x}\in c_{0,\mathcal I}\setminus c_0$ and ${\bf0}<|{\bf y}|\leq|{\bf x}|$.

(2) $\Rightarrow$ (3): Suppose there is ${\bf x}\in c_{0,\mathcal I^\perp}$ such that  ${\bf x}\not\in c_0$. Since $c_{0,\mathcal I}\setminus c_0$ is order dense in $\ell_\infty\setminus c_{0}$, there exists ${\bf y}\in c_{0,\mathcal I}\setminus c_0$ satisfying $0<|{\bf y}|\leq|{\bf x}|$.
Thus, $A(\varepsilon,{\bf y})\subset A(\varepsilon,{\bf x})$ for each $\varepsilon>0$. There is  $\varepsilon_0>0$ with $A(\varepsilon_0,{\bf x})$ infinite. Since $A(\varepsilon_0,{\bf x})\in \ideal^\perp$ and $A(\varepsilon,{\bf y})\in \ideal$,  it follows that $A(\varepsilon,{\bf y})$ is finite for each $0<\varepsilon<\varepsilon_0$ which contradicts that ${\bf y}\not\in c_0$. So, $c_{0,\mathcal I^\perp}=c_0$.

(3) $\Rightarrow$ (1): The equality $c_{0,\mathcal I^\perp}=c_0$ means that $\mathcal I^\perp=\sf{Fin}$. Thus, $\mathcal I$ is tall.
\end{proof}

\section{Banach spaces isomorphic to $c_{0,\ideal}(X)$}

In this section we show that, for some ideals, $c_{0,\ideal}$ is isomorphic to a known classical Banach space. We do it for ideals of the form $\mathcal I\sqcup\mathcal J$,
$\bigoplus\limits_{n\in\mathbb N} \ideal_n$, $\ideal^{\omega\perp}$ and the Fubini product $\ideal\times\idealj$.

\begin{theorem}
\label{union of ideals}
Let $X$ be a Banach space. If $\mathcal I$ and $\mathcal J$ are ideals on $\N$, then $c_{0,\mathcal I\sqcup\mathcal J}(X)=c_{0,\mathcal I}(X)+c_{0,\mathcal J}(X)$.
\end{theorem}

\begin{proof}
If ${\bf x}=(x_n)\in c_{0,\mathcal I}(X)$ and ${\bf y}=(y_n)\in c_{0,\mathcal J}(X)$, then $A(\varepsilon,{\bf x+y})\subset A(\varepsilon/2,{\bf x})\cup A(\varepsilon/2,{\bf y})$ for all $\varepsilon>0$. So, $c_{0,\mathcal I}(X)+c_{0,\mathcal J}(X)\subset c_{0,\mathcal I\sqcup\mathcal J}(X)$. On the other hand, if ${\bf z}=(z_n)\in\ell_\infty(X)$ and $\mathrm{supp}({\bf z})\in\mathcal I\sqcup\mathcal J$, take $A\in\mathcal I$ and $B\in\mathcal J$ with $\mathrm{supp}({\bf z})=A\cup B$ and $A\cap B=\emptyset$. By setting ${\bf x}=(\chi_A(n)z_n)$ and ${\bf y}=(\chi_B(n)z_n)$, we have ${\bf z= x+y}$, ${\bf x}\in c_{0,\mathcal I}$ and ${\bf y}\in c_{0,\mathcal J}$. By Proposition \ref{span of I}(1) we conclude that $c_{0,\mathcal I\sqcup\mathcal J}(X)\subset c_{0,\mathcal I}(X)+c_{0,\mathcal J}(X)$.
\end{proof}

\begin{theorem}
\label{directsum}
Let $X$ be a Banach space, $\{K_n:\; n\in\mathbb N\}$ be a partition of $\mathbb N$ and $\ideal_n$ be an ideal on $K_n$ for each $n\in\mathbb N$. If  $\mathcal I=\bigoplus\limits_{n\in\mathbb N} \ideal_n$, then  $c_{0,\mathcal I}(X)$ is isometric to $\ell_\infty((c_{0,\mathcal I_n}(X))_{n\in\mathbb N})$. In particular,  $c_{0,\mathcal{J}^{\omega}}(X)$ is isometric to $\ell_\infty(c_{0,\mathcal{J}}(X))$ for any ideal $\mathcal{J}$.

\end{theorem}

\begin{proof}
    Let 
\begin{align*}
    \Psi\colon c_{0,\mathcal I}(X)&\longrightarrow \ell_\infty((c_{0,\mathcal I_n}(X))_{n\in\mathbb N})\\
    (x_n)&\longmapsto((x_n)_{n\in K_m})_{m\in\mathbb N}.
\end{align*}
\begin{enumerate}
\item To see that $\Psi$ is a  well defined linear isometry, let ${\bf x}=(x_n)\in c_{0,\mathcal I}(X)$ and $\varepsilon>0$. Then 
\begin{gather*}
A(\varepsilon,{\bf x}) \in\mathcal I\quad\Longleftrightarrow\mbox{$ \{n\in K_m\,\colon\,\|x_n\|\geq\varepsilon\}\in\mathcal I_m$  for each $m\in\mathbb N$.}
\end{gather*}
Thus, $(x_n)_{n\in K_m}\in c_{0,\mathcal I_m}(X)$ for all $m\in\mathbb N$. On the other hand,
\begin{gather*}
\|\Psi({\bf x})\|=\sup_{m\in\mathbb N}\|(x_n)_{n\in K_m}\|_\infty=\sup_{m\in\mathbb N}\sup_{n\in K_m}\|x_n\|=\sup_{n\in\mathbb N}\|x_n\|=\|{\bf x}\|.
\end{gather*}

\item To see that $\Psi$ is onto,  let $\displaystyle{\bf y}=((y_n^m))_{m\in\mathbb N}\in \ell_\infty((c_{0,\mathcal I_m}(X))_{m\in\mathbb N})$ be given. Define ${\bf x}=(x_n)\in\ell_\infty(X)$ by $x_n=y_n^m$ iff $n\in K_m$. If $\varepsilon>0$ and $m\in\mathbb N$ are given, then
    \begin{gather*}
        \{n\in\mathbb N\,\colon\,\|x_n\|\geq\varepsilon\}\cap K_m= \{n\in K_m\,\colon\,\|y_n^m\|\geq\varepsilon\}\in\mathcal I_m.
    \end{gather*}
    So $ A(\varepsilon,{\bf x}) \in\mathcal I$ for each $\varepsilon>0$ i.e. ${\bf x}\in c_{0,\mathcal I}(X)$. Clearly $\Psi({\bf x})={\bf y}$.\qedhere
\end{enumerate}
\end{proof}

\begin{corollary}
Let $X$ be a Banach space.  Let $\{A,B\}$ be a partition of $\mathbb N$ and  $\mathcal I$ and $\mathcal J$ be ideals on $A$ and $B$, respectively. Then $c_{0,\mathcal I\oplus\mathcal J}(X)$ is isometric to $c_{0,\mathcal I}(X)\oplus_{\infty} c_{0,\mathcal J}(X)$.
\end{corollary}

\begin{theorem}
\label{omegaperp}
Let $X$ be a Banach space, $\{K_n:\; n\in\mathbb N\}$ be a partition of $\mathbb N$,  $\ideal_n$ be an ideal on $K_n$ for each $n\in\mathbb N$ and   $\ideal=\bigoplus\limits_{n\in\mathbb N} \ideal_n$.  Then
$c_{0,\mathcal I^{\perp}}(X)$ is isometric to $c_0((c_{0,\mathcal I_j^\perp}(X))_{j\in\mathbb N})$. In particular,  $c_{0,\mathcal{J}^{\omega\perp}}(X)$ is isometric to $c_0(c_{0,\mathcal{J}^\perp}(X))$ for any ideal $\mathcal{J}$.
\end{theorem}

\begin{proof}
 Recall that 
\begin{gather*}
A\in \ideal^{\perp}\Longleftrightarrow\mbox{$(\exists N\in\N)(A\subseteq \bigcup_{i\leq N} K_i$ and $(\forall i\leq N)\; (A\cap K_i\in \ideal_{i}^{\perp}))$.}
\end{gather*}
For any ${\bf x}\in c_{0,\mathcal I^{\perp}}(X)$, we set ${\bf y}^j=(x_n)_{n\in K_j}$ for each $j\in\mathbb N$. We have the following:
\begin{enumerate}
\item ${\bf y}^j\in c_{0,\ideal_j^\perp}(X)$ for all $j\in\mathbb N$. In fact, let  $j\in\mathbb N$  and $\delta>0$ be given.  There is $N\in\mathbb N$ such that $A(\delta,{\bf x})\subset K_1\cup\cdots\cup K_N$ and 
$A(\delta,{\bf x})\cap K_i\in\mathcal I_i^{\perp}$ for all $i\in\{1,\ldots,N\}$. Observe that $A(\delta,{\bf x})\cap K_j=A(\delta,{\bf y}^j)$. Thus, if $j\leq N$, we have that $A(\delta,{\bf y}^j)\in \mathcal I_i^{\perp}$. Otherwise,  $A(\delta,{\bf y}^j)=A(\delta,{\bf x})\cap K_j=\emptyset$. In either case, we conclude that ${\bf y}^j\in c_{0,\mathcal I_j^{\perp}}(X)$. 

\item The sequence $({\bf y}^j)_{j\in\mathbb N}$ converges to 0. If $\varepsilon>0$ is given, there is $N\in\mathbb N$ such that $A(\varepsilon,{\bf x})\subset K_1\cup\cdots\cup K_N$. If $j>N$ and $n\in K_j$, we have $\|x_n\|<\varepsilon$. So $\displaystyle\|{\bf y}^j\|=\sup_{n\in K_j}\|x_n\|\leq\varepsilon$ if $j>N$.
\end{enumerate}
The two previous claims show that 
\begin{align*}
    \Psi\colon c_{0, \mathcal I^{\perp}}(X)&\longrightarrow c_0((c_{0,\mathcal I_j^{\perp}}(X))_{j\in\mathbb N})\\
    {\bf x}=(x_n)&\longmapsto({\bf y}^j)_{j\in\mathbb N}
\end{align*}
is well defined and is clearly  a linear isometry.

To see that $\Psi$ is onto,  let $({\bf y}^j)\in c_0((c_{0,\mathcal I_j^\perp}(X))_{j\in\mathbb N})$ be given and consider the sequence ${\bf x}=(x_n)_{n}$ defined by $x_n=y^j_n$ if $n\in K_j$. Let $\varepsilon>0$, there is $N\in\mathbb N$ such that 
$\|{\bf y}^j\|=\displaystyle\sup_{n\in K_j}\|y_n^j\|<\varepsilon$ for all $j>N$. Thus,
$A(\varepsilon,{\bf x})\subset K_1\cup\cdots\cup K_N$. Also note that
$A(\varepsilon,{\bf x})\cap K_j=\{n\in K_j\,\colon\,\|y_n^j\|\geq\varepsilon\}\in \ideal_j^{\perp}$ for all $j\leq N$, which yields $(x_n)\in c_{0,\mathcal I^{\perp}}(X)$. Clearly, $\Psi({\bf x})=({\bf y}^j)$.
\end{proof}

To study the Banach space $c_{0,\mathcal I\times\mathcal J}(X)$ we  need to compute the norm of the quotient  $\ell_\infty/c_{0,\mathcal I}$.
For that end  we recall the definition of $\mathcal I-\limsup$ of a sequence.

\begin{definition}
\cite{demirci}
Let $\mathcal I$ be an ideal on $\N$. For each sequence ${\bf x}=(x_n)$  in $\mathbb R$, we set
\begin{gather*}
B_{\bf x}=\{b\in\mathbb R\,\colon\,\{k\in\mathbb N\,\colon\,x_k>b\}\not\in\mathcal I\}
\end{gather*}
and 
\begin{gather*}
 \mathcal I-\limsup x_n=
 \begin{cases}
\sup B_{\bf x},& B_{\bf x}\neq\emptyset,\\
-\infty, & B_{\bf x}=\emptyset.
 \end{cases}
 \end{gather*}
\end{definition}

Note that a sequence $(x_n)$ is $\mathcal I$-convergent to 0 iff
$\mathcal I-\limsup|x_n|=0$ (see \cite[Theorem 4]{demirci}).

\begin{lemma}
\label{properties of I-limsup}
Let $\mathcal I$ be an ideal on $\N$.
\begin{enumerate}
\item If $x_n\leq y_n$ for each $n\in\mathbb N$, then $\mathcal I-\limsup x_n\leq\mathcal I-\limsup y_n$.

\item If $c\in\mathbb R$, then $\mathcal I-\limsup c+x_n=c+\mathcal I-\limsup x_n$.
\end{enumerate}
\end{lemma}

\begin{proof}
(1) If $\mathcal I-\limsup x_n=-\infty$, there is nothing to prove. Suppose that $B_{\bf x}\neq\emptyset$ and let $\varepsilon>0$ be given. If $b\in B_{\bf x}$ satisfies $\mathcal I-\limsup x_n-\varepsilon<b$, then $\{k\in\mathbb N\,\colon\,y_k>b\}\not\in\mathcal I$. Hence $b\in B_{\bf y}$ and $\mathcal I-\limsup x_n-\varepsilon\leq\mathcal I-\limsup y_n$. Since  $\varepsilon>0$ was arbitrary, the result follows.

(2) Just observe that 
\begin{gather*}
\{b\in\mathbb R\,\colon\,\{k\in\mathbb N\,\colon\,x_k+c>b\}\not\in\mathcal I\}=c+\{b'\in\mathbb R\,\colon\,\{k\in\mathbb N\,\colon\,x_k>b'\}\not\in\mathcal I\}.
\qedhere
\end{gather*}
\end{proof}

\begin{lemma}
\label{quotient norm of l/cI}
Let $X$ be a Banach space and  ${\bf x}=(x_n)\in\ell_\infty(X)$. Then 
    \begin{gather*}
        \mathcal I-\limsup\|x_n\|=\|(x_n)+c_{0,\mathcal I}(X)\|.
    \end{gather*} 
\end{lemma}

\begin{proof}
Let ${\bf z}\in c_{0,\mathcal I}(X)$. Then $\|x_n\|\leq\|{\bf x-z}\|+\|z_n\|$ for all $n\in\mathbb N$. By Lemma \ref{properties of I-limsup} we have $\mathcal I-\limsup\|x_n\|\leq\|{\bf x-z}\|$. Thus, $\mathcal I-\limsup\|x_n\|\leq\inf\{\|{\bf x-z}\|\,\colon\,{\bf z}\in c_{0,\mathcal I}(X)\}$. Now let $r\in\mathbb R$ be such that 
    \begin{gather*}
        \mathcal I-\limsup\|x_n\|<r<\inf\{\|{\bf x-z}\|\,\colon\,{\bf z}\in c_{0,\mathcal I}(X)\}.
    \end{gather*}
Note that $A:=A(r,{\bf x})\in\mathcal I$. If $A=\emptyset$, we are done. If not, for $n\in A$ we have ${\bf z}=x_n\chi_{\{n\}}\in c_{0,\mathcal I}(X)$ and $r<\|{\bf x-z}\|=0$, which is impossible. Whence, $\mathcal I-\limsup\|x_n\|=\inf\{\|{\bf x-z}\|\,\colon\,{\bf z}\in c_{0,\mathcal I}(X)\}=\|(x_n)+c_{0,\mathcal I}(X)\|$.
\end{proof}

\begin{theorem}
\label{c_0,I for Fubini products}
Let $X$ be a Banach space and $\mathcal I$ and $\mathcal J$ be ideals on $\mathbb N$.  Then $c_{0,\mathcal I\times\mathcal J}(X)/c_{0,{\mathcal J}^{\omega}}(X)$ is isomorphic to $c_{0,\mathcal I}(\ell_\infty(X)/c_{0,\mathcal J}(X))$.
\end{theorem}

\proof
  We start by proving:
     \begin{claim}\label{relationships between I-limits}
         Let ${\bf x}=(x_n)\in\ell_\infty(X)$ and $\delta>0$ be given.
\begin{enumerate}
    \item If $A(\delta,{\bf x})\in\mathcal I$, then $\mathcal I-\limsup_j\|x_j\|\leq\delta$.
    \item If $\mathcal I-\limsup_j\|x_j\|\leq\delta$, then $A(2\delta,{\bf x})\in\mathcal I$.
\end{enumerate}
     \end{claim}

\begin{proof}[Proof of claim.] (1): Suppose that $\mathcal I-\limsup_{j}\|x_j\|>\delta$. There is $b\in\mathbb R$ such that $\delta<b$ and
$\{j\in\mathbb N\,\colon\,\|x_j\|>b\}\not\in\mathcal I$. But we have $\{j\in\mathbb N\,\colon\,\|x_j\|>b\}\subset A(\delta,{\bf x})\in\mathcal I$, which is absurd. So, $\mathcal I-\limsup_{j}\|x_j\|\leq\delta$.

   (2): If $A(2\delta,{\bf x})=\{j\in\mathbb N\,\colon\,\|x_j\|\geq 2\delta\}\not\in\mathcal I$, we have $\{j\in \mathbb N\,\colon\,\|x_j\|>\delta'\}\not\in\mathcal I$ for all $\delta'<2\delta$. Hence, $\delta'\leq\mathcal I-\limsup_{j}\|x_j\|$. Since $\delta'<2\delta$ were arbitrary, we conclude that  $2\delta\leq\mathcal I-\limsup_{j}\|x_j\|$, a contradiction.
\end{proof}

Now we continue with the proof of the theorem. If ${\bf x}=(x_{(n,m)})\in\ell_\infty(\mathbb N\times\mathbb N)(X)$, we set
${\bf x}_{(n)}=(x_{(n,m)})_{m\in\mathbb N}$. Let
\begin{align*}
    \Psi\colon c_{0,\mathcal I\times\mathcal J}(X)&\to c_{0,\mathcal I}(\ell_\infty(X)/c_{0,\mathcal J}(X))\\
    {\bf x}=(x_{(n,m)})&\mapsto({\bf x}_{(n)}+c_{0,\mathcal J}(X))_{n\in\mathbb N}.
\end{align*}
From Claim \ref{relationships between I-limits} and Proposition \ref{quotient norm of l/cI} we have
\begin{align*}
    {\bf x}=(x_{(n,m)})\in c_{0,\mathcal I\times\mathcal J}(X)
    &\Longleftrightarrow(\forall\,\varepsilon>0)\,
    (\{n\in\mathbb N\,\colon\,\{m\in\mathbb N\,\colon\,(n,m)\in A(\varepsilon,{\bf x})\}\not\in\mathcal J\}\in\mathcal I)\\
    &\Longleftrightarrow(\forall\,\varepsilon>0)\,
    (\{n\in\mathbb N\,\colon\,\{m\in\mathbb N\,\colon\,\|x_{(n,m)}\|\geq\varepsilon\}\not\in\mathcal J\}\in\mathcal I)\\
    &\Longleftrightarrow(\forall\,\varepsilon>0)\,
    (\{n\in\mathbb N\,\colon\,\mathcal J-\limsup\|{\bf x}_{(n)}\|\geq\varepsilon\}\in\mathcal I)\\
    &\Longleftrightarrow(\forall\,\varepsilon>0)\,
    (\{n\in\mathbb N\,\colon\,\|{\bf x}_{(n)}+c_{0,\mathcal J}(X)\|\geq\varepsilon\}\in\mathcal I)\\
    &\Longleftrightarrow({\bf x}_{(n)}+c_{0,\mathcal J}(X))_{n\in\mathbb N}\in c_{0,\mathcal I}(\ell_\infty(X)/c_{0,\mathcal J}(X)).
\end{align*}
So, $\Psi$ is well defined. Now, if $({\bf y}^n+c_{0,\mathcal J}(X))_{n\in\mathbb N}\in c_{0,\mathcal I}(\ell_\infty(X)/c_{0,\mathcal J}(X))$ is given, we set
$x_{(n,m)}=y_m^n$ for $n,m\in\mathbb N$. It is not difficult to check that ${\bf x}=(x_{(n,m)})\in c_{0,\mathcal I\times\mathcal J}(X)$ and $\Psi({\bf x})=({\bf y}^n+c_{0,\mathcal J}(X))_{n\in\mathbb N}$. Now, 
\begin{gather*}
   \|\Psi({\bf x})\|=\sup_{n\in\mathbb N}\left(\mathcal J-\limsup\|{\bf x}_{(n)}\|\right)\leq\sup_{n\in\mathbb N}\sup_{m\in\mathbb N}\|x_{(n,m)}\|
   =\|{\bf x}\|_\infty.
\end{gather*}
Also we have
\begin{align*}
    \Psi({\bf x})={\bf0}&\Longleftrightarrow (\forall\, n\in\mathbb N)\, (\,{\bf x}_{(n)}\in c_{0,\mathcal J}(X)\,)\\
    &\Longleftrightarrow (\forall\,\varepsilon>0)\,(\forall\,n\in\mathbb N)\,(\,A(\varepsilon,{\bf x}_{(n)})\in{\mathcal J}\,)\\
    &\Longleftrightarrow (\forall\,\varepsilon>0)\, (\,A(\varepsilon,{\bf x})\in{\mathcal J}^\omega\, )\\
    &\Longleftrightarrow{\bf x}\in c_{0,{\mathcal J}^\omega}(X).\qedhere
\end{align*}
\endproof

To state the next result,  we recall that if $X$ and $Y$ are Banach spaces, $X\stackrel{\vee}{\otimes} Y$ denotes the \textit{injective tensor product} of $X$ and $Y$ (see \cite[Chapter 1]{diestel and all}).

\begin{theorem}
\label{tensor products}
Let $X$ be  a Banach space and $\mathcal I$ be an ideal on $\N$. Then $c_{0,\mathcal I}\stackrel{\vee}{\otimes} X$ is isometric to a subspace of $c_{0,\mathcal I}(X)$.
\end{theorem}

\begin{proof}
If $u=\sum_{j=1}^m(x_n^j)_n\otimes y_j\in c_{0,\mathcal I}\otimes X$, then 
\begin{align*}
        |u|_{\vee}&=\sup_{x^*\in B_{c_{0,\mathcal I}^*},y^*\in B_{X^*}}\left|\sum_{j=1}^mx^*((x_n^j)_n)y^*(y_j)\right|\\
        &=\sup_{x^*\in B_{c_{0,\mathcal I}^*},y^*\in B_{X^*}}\left|x^*\left(\sum_{j=1}^m(x_n^jy^*(y_j))_n\right)\right|\\
        &=\sup_{y^*\in B_{X^*}}\left\|\sum_{j=1}^m(x_n^jy^*(y_j))_n\right\|\\
        &=\sup_{y^*\in B_{X^*}}\left\|\sum_{j=1}^m(y^*(x_n^jy_j))_n\right\|\\
        &=\sup_{y^*\in B_{X^*}}\left\|\left(y^*\left(\sum_{j=1}^mx_n^jy_j\right)\right)_n\right\|\\ &=\left\|\left(\sum_{j=1}^mx_n^jy_j\right)_n\right\|.
\end{align*}
So the map 
\begin{gather*}
\sum_{j=1}^m(x_n^j)_n\otimes y_j\in c_{0,\mathcal I}\otimes X\mapsto\left(\sum_{j=1}^mx_n^jy_j\right)_n\in c_{0,\mathcal I}(X)
\end{gather*}
extends to an isometry from $c_{0,\mathcal I}\stackrel{\vee}{\otimes} X$ to $c_{0,\mathcal I}(X)$. 
\end{proof}

\begin{remark}
The above isometry is not onto in general. Indeed, if $\mathcal I=\mathcal P(\mathbb N)$, then  $c_{0,\mathcal  I}$ is $\ell_\infty$ and it is known that the existence of an isometry from $\ell_\infty\stackrel{\vee}{\otimes} X$ onto $\ell_\infty(X)$ implies that $X$ contains a complemented copy of $c_0$ (see \cite{leung-rabiger}).
However, when $\ideal=\fin$, the above isometry is onto as it is established in \cite[Theorem 1.1.11]{diestel and all}.
\end{remark}

\begin{question}
For which ideals  $\ideal$, is $c_{0,\mathcal I}\stackrel{\vee}{\otimes} X$  isometric to $c_{0,\mathcal I}(X)$? 
\end{question}

\subsection{Some  examples}
\label{examples}

In \cite{GU2018} was studied the smallest collection
$\mathcal{B}$ of ideals on $\N$ containing the ideal of finite sets and
closed under countable direct sums and the operation of taking orthogonal.   They  defined by recursion a sequence of ideals $P_\alpha$ and $Q_\alpha$ for $\alpha<\omega_1$. 
For every limit ordinal $\alpha<\omega_1$ we fix an increasing sequence $(\upsilon_{n}^{\alpha})_n$ of ordinals  such that
$\sup_n(\upsilon_{n}^{\alpha})=\alpha$.

\begin{itemize}
\item[(i)] $P_0=\mathcal{P}(\mathbb{N})$ and $Q_0=(P_{0})^{\perp}=\fin$.
\item[(ii)] $P_{\alpha+1}=(P_{\alpha}^{\perp})^\omega$.
\item[(iii)]  $P_{\alpha}=\oplus_{n}(P_{\upsilon_{n}^{\alpha}})^\perp$,  for $\alpha<\omega_1$ a limit ordinal.
\item[(iv)]  $Q_\alpha=(P_\alpha)^\perp$ for every $\alpha<\omega_1$.
\end{itemize}
Then  $\mathcal{B}=\{P_\alpha, Q_\alpha: \alpha<\omega_1\}$  has exactly $\aleph_1$ non isomorphic  ideals.

For example,  $P_1=\fin^\omega$ is the  sum countably many times the ideal $\fin$ of all finite subsets of $\N$, it is a well known ideal  sometimes  denoted by $\{\emptyset\}\times \fin$. Its orthogonal $Q_1=\fin^{\omega\perp}$ is  also denoted  $\fin\times \{\emptyset\}$.  All these ideals are Fr\'echet (that is, they satisfy that $\ideal^{\perp\perp}=\ideal$). In particular,  $(Q_\alpha)^\perp=P_\alpha$. 

\begin{theorem} \cite{GU2018} 

(i) Every ideal in $\mathcal{B}$ is isomorphic to either $P_\alpha$, $Q_\alpha$ or $P_\alpha\oplus Q_\alpha$ for some $\alpha<\omega_1$.

(ii) For every $\alpha<\beta<\omega_1$, there are subsets $A,B$ of $\N$ such that $P_\alpha\approx Q_\beta\restriction A$ and $Q_\alpha\approx P_\beta\restriction B$.

(iii) All ideals in $\mathcal{B}$ are Borel subsets of $2^\N$, and thus they are meager ideals. 
\end{theorem}

Let  ${\sf WO}(\Q)$ denote the ideal of all  well founded subsets of ${\sf WO}(\Q)$.  For simplicity, we will write ${\sf WO}$ instead of ${\sf WO}(\Q)$. Observe that ${\sf WO}^\perp$ is the ideal of well
founded subsets of $(\Q, <^*)$ where $<^*$ is the reversed order of $\Q$. In fact, the map $x\mapsto -x$ from $\Q$ onto $\Q$ is an isomorphism between ${\sf WO}$ and $WO^\perp$. In particular, ${\sf WO}$ is a
Fr\'{e}chet ideal. We recall also that ${\sf WO}$ is a meager ideal since it is a co-analytic  subset of the Cantor cube $\{0,1\}^\Q$. 
 From the results in \cite{GU2018} we have

\begin{theorem}
\label{ByWO}
Every ideal $\ideal$ in $\mathcal{B}$ is isomorphic to a restriction ${\sf WO}\restriction A$ for some $A\subseteq \Q$.
\end{theorem}

\begin{theorem}
\label{the banach spaces c_0,a} 
The following hold for every countable ordinal $\alpha$.
 \begin{itemize}
     \item[(i)] $c_{0,P_{\alpha+1}}$ is isometric to $\ell_\infty (c_{0,Q_{\alpha}})$.

\item[(ii)] $c_{0,Q_{\alpha+1}}$ is isometric to $c_0 (c_{0,P_{\alpha}})$.

\item[(iii)] $c_{0,P_{\alpha}}$ is isometric to $\ell_\infty ((c_{0,Q_{\upsilon_n^\alpha}})_n)$ for $\alpha$ a limit ordinal.

\item[(iv)] $c_{0,Q_{\alpha}}$ is isometric to $c_0((c_{0,P_{\upsilon_n^\alpha}})_n)$ for $\alpha$ a limit ordinal.

 \end{itemize}   
\end{theorem}

\begin{proof}  (i) Since  $P_{\alpha+1}=(Q_{\alpha})^\omega$, the claim follows from Theorem \ref{directsum}. (ii) As $Q_{\alpha+1}=(Q_{\alpha})^{\omega\perp}$ and $(Q_\alpha)^\perp=P_\alpha$, the claim follows from Theorem \ref{omegaperp}.
(iii) Let $\alpha$ be a limit ordinal. The result follows from Theorem \ref{directsum} as  $P_{\alpha}=\oplus_{n}(P_{\upsilon_{n}^{\alpha}})^\perp=\oplus_{n}Q_{\upsilon_{n}^{\alpha}}$. (iv) Since $Q_\alpha=(P_\alpha)^\perp=(\oplus_n(P_{\upsilon_{n}^{\alpha}})^\perp)^\perp$, the result follows from Theorem \ref{omegaperp}.
\end{proof}

From Proposition \ref{restriction} and Theorem \ref{ByWO} we immediately get the following.  Observe that $c_{0,{\sf WO}}$ is not isomorphic to $\ell_\infty$ by Theorem \ref{Teorema de leonetti}.

\begin{theorem}
 For every  $\alpha<\omega_1$,  $c_{0,P_{\alpha}}$ and $c_{0,Q_{\alpha}}$ are isometric to a closed subspace of $c_{0,{\sf WO}}$. 
\end{theorem}

Some particular instances of Theorem \ref{the banach spaces c_0,a} are  the following:

\begin{itemize}
\item[] $c_{0,\fin}=c_0$.

\item[] $c_{0,\fin^\perp}=\ell_\infty$.

\item[] $c_{0,\fin^{\perp\omega}}$ is isometric to $\ell_\infty (\ell_\infty)=\ell_\infty$.

\item[] $c_{0, \fin^{\omega}}$ is isometric to $\ell_\infty (c_0)$.
    
\item[] $c_{0, \fin^{\omega\perp}}$ is isometric to $c_0(\ell_\infty )$.

 \item[] $c_{0, \fin^{\omega\perp\omega}}$ is isometric to $\ell_\infty (c_0(\ell_\infty))$. 

\item[]  $c_{0, \fin^{\omega\perp\omega\perp}}$ is isometric to $c_0(\ell_\infty (c_0))$.

\end{itemize}

\medskip

Cembranos and Mendoza \cite{cembranos-mendoza2010} showed that 
$\ell_\infty(c_0)$  and  $c_0(\ell_\infty)$ are not isomorphic.  
Since $\fin^{\omega}$ and $\fin^{\omega\perp}$ are not isomorphic \footnote{A quick way to see this is  noticing that $\fin^{\omega\perp}$ is $F_\sigma$ and $\fin^{\omega}$ is not.}, from Theorem \ref{c_0,I caracteriza ideales}, we obtain the following  weak version of their result.

\begin{theorem}
\cite{cembranos-mendoza2010}
$\ell_\infty(c_0)$ is not isometric to $c_0(\ell_\infty)$.     
\end{theorem}

Concerning to the Banach space $c_{0,{\sf Fin\times Fin}}$, Theorem \ref{c_0,I for Fubini products} implies that $c_0(\ell_\infty/c_0)$ is isomorphic to $c_{0,\sf{FIN}\times\sf{FIN}}/c_{0,\sf{FIN}^\omega}$. On the other hand, the following result gives a quotient representation of $c_{0,{\sf Fin\times FIN}}$.

\begin{theorem}
$c_{0,\fin\times\fin}$ is isomorphic to $\ell_\infty(c_0)\times c_0(\ell_\infty)/K$, where $K$ is isomorphic to $c_0$.
\end{theorem}

\begin{proof}
Consider the maps $\Psi_1\colon c_{0,{\sf Fin^{\omega}}}\to \ell_\infty((c_0(K_n))_{n\in\mathbb N})$ and $\Psi_2\colon c_{0,{\sf Fin^{\omega\perp}}}\to c_0((\ell_\infty(K_n))_{n\in\mathbb N})$ defined in Theorems \ref{directsum} and \ref{omegaperp}, respectively. Let $S\colon c_{0,{\sf Fin^{\omega}}}\times c_{0,{\sf Fin^{\omega\perp}}}\to c_{0,{\sf Fin^{\omega}}}+c_{0,{\sf Fin^{\omega\perp}}}$ be defined by $S(a,b)=a+b$, for $a\in c_{0,{\sf Fin^{\omega}}}$ and $b\in c_{0,{\sf Fin^{\omega\perp}}}$. If $\Psi_1^{-1}\times\Psi_2^{-1}$ is the natural map from $\ell_\infty((c_0(K_n))_{n\in\mathbb N})\times c_0((\ell_\infty(K_n))_{n\in\mathbb N})$ onto $c_{0,{\sf Fin^{\omega}}}\times c_{0,{\sf Fin^{\omega\perp}}}$, then 
\begin{gather*}
    S\circ(\Psi_1^{-1}\times\Psi_2^{-1})\colon (x,y)\in \ell_\infty((c_0(K_n))_{n\in\mathbb N})\times c_0((\ell_\infty(K_n))_{n\in\mathbb N})\to \Phi_1^{-1}(x)+\Phi_2^{-1}(y)\in c_{0,{\sf Fin^{\omega}}}+c_{0,{\sf Fin^{\omega\perp}}}
\end{gather*}
is linear, continuous and onto. By definition of $\fin\times\fin$ we have 
\[
\fin\times\fin=\{A\cup B:\;A\in \fin^{\omega},\; B\in \fin^{\omega\perp} \}.
\]
So, $c_{0,{\sf Fin\times FIN}}=c_{0,{\sf Fin}^{\omega}}+c_{0,{\sf Fin}^{\omega\perp}}$ by Theorem \ref{union of ideals}. From the fact $\ell_\infty((c_0(K_n))_{n\in\mathbb N})$ and $c_0((\ell_\infty(K_n))_{n\in\mathbb N})$ are isometric to $\ell_\infty(c_0)$ and $c_0(\ell_\infty)$, respectively, we get the first part of the result. Finally, note that  $K=\ker S\circ(\Psi_1^{-1}\times\Psi_2^{-1})$  is isomorphic to $c_0(c_0)$. Indeed, we have 
\begin{align*}
    K&=(\Psi_1^{-1}\times\Psi_2^{-1})^{-1}(\ker S)\\
&=(\Psi_1^{-1}\times\Psi_2^{-1})^{-1}(\{(a,-a)\,\colon\,a\in c_{0,{\sf Fin^{\omega}}}\cap c_{0,{\sf Fin^{\omega\perp}}}\})\\
&=(\Psi_1^{-1}\times\Psi_2^{-1})^{-1}(\{(a,-a)\,\colon\,a\in c_{0,{\sf Fin}}\})\\
&=\{(x,-x)\,\colon\,x\in c_0((c_0(K_n))_{n\in\mathbb N})\}.
\end{align*}
So, the map $x\in c_0((c_0(K_n))_{n\in\mathbb N})\mapsto (x,-x)\in K$ is an isomorphism.   
\end{proof}


\subsection{Grothendieck property}

Recall that a Banach space is \textit{Grothendieck} (or has the {\em  Grothendieck property}) if every operator from $X$ to $c_0$ is weakly compact. It is known that the Grothendieck property pass to quotients and complemented subspaces \cite[Proposition 3.1.4]{gonzalez-kania}. 

Now we are interested in the Grothendieck property for the Banach spaces 
$c_{0,\mathcal I}$. In \cite[Problem 9]{gonzalez-kania} it is asked about the ideals
$\mathcal I$ such that $c_{0,\mathcal I}$ is Grothendieck. Concerning to this question, we have the following results.

\begin{theorem}
    For each ordinal $1\leq\alpha<\omega_1$, $c_{0,P_\alpha}$ and $c_{0,Q_\alpha}$ are not Grothendieck spaces.
\end{theorem}

\begin{proof}
We proceed by induction on $\alpha$. If $\alpha=1$, $c_{0,P_1}$ and $c_{0,Q_1}$ are isometric to $\ell_\infty(c_0)$ and $c_0(\ell_\infty)$, respectively. Note that $c_{0,P_1}$ is not  Grothendieck because $c_0$ is a quotient of it. On the other hand, $c_{0,Q_1}$ is not Grothendieck by \cite[Proposition 5.3.8]{gonzalez-kania} since $c_0(\ell_\infty)$ is isometric to $c_0\stackrel{\vee}{\otimes}\ell_\infty$ \cite[Theorem 1.1.11]{diestel and all}.

Now, let $\beta<\omega_1$ and suppose that 
$c_{0,P_\beta}$ and $c_{0,Q_\beta}$ are not Grothendieck.  By Theorem \ref{the banach spaces c_0,a}, $c_{0,P_{\beta+1}}$ and $c_{0,Q_{\beta+1}}$ are isometric to $\ell_\infty(c_{0,Q_\beta})$ and 
 $c_0(c_{0,P_\beta})$, respectively. Since $c_{0,Q_\beta}$ and $c_{0,P_\beta}$ are quotient of $\ell_\infty(c_{0,Q_\beta})$ and 
$c_0(c_{0,P_\beta})$, respectively, we conclude that $c_{0,P_{\beta+1}}$ and $c_{0,Q_{\beta+1}}$ do not have the Grothendieck property.  Now, let $\alpha$ be a limit ordinal and assume that $c_{0,P_\beta}$ and $c_{0,Q_\beta}$ do not have the Grothendieck property for all $\beta<\alpha$. Fix an increasing sequence $(\upsilon_{n}^{\alpha})_n$ of ordinals  such that
$\sup_n(\upsilon_{n}^{\alpha})=\alpha$. By Theorem \ref{the banach spaces c_0,a}, $c_{0,P_{\alpha}}$ and 
    $c_{0,Q_{\alpha}}$ are isometric to $\ell_\infty ((c_{0,Q_{\upsilon_n^\alpha}})_n)$ and $c_0((c_{0,P_{\upsilon_n^\alpha}})_n)$, respectively. Since $c_{0,Q_{\upsilon_n^\alpha}}$ and $c_{0,P_{\upsilon_n^\alpha}}$ do not have  the Grothendieck property for any $n\in\mathbb N$, we infer that $\ell_\infty ((c_{0,Q_{\upsilon_n^\alpha}})_n)$ and $c_0((c_{0,P_{\upsilon_n^\alpha}})_n)$ are not Grothendieck spaces.
\end{proof}

\begin{theorem}
\label{GP}
Let $\mathcal I$ be an  ideal on $\N$. If $\mathcal I$ is a meager, then $c_{0,{\sf Fin}\times\mathcal I}$ is not a Grothendieck space.
\end{theorem}

\begin{proof}
Suppose that $c_{0,{\sf Fin}\times\mathcal I}$ is Grothendieck. By  Theorem \ref{c_0,I for Fubini products}, $c_0(\ell_\infty/c_{0,\mathcal I})$ is Grothendieck. By \cite[Proposition 5.3.8]{gonzalez-kania}, $\ell_\infty/c_{0,\mathcal I}$ must by finite-dimensional. Consequently,  $c_{0,\mathcal I}$ is complemented in $\ell_\infty$. By Theorem \ref{Teorema de leonetti}, $\mathcal I$ is a non-meager ideal.
\end{proof}

It is known that $c_{0,\ideal}$ is complemented if $\ideal$ is a maximal ideal (\cite{Leonetti2018}). A more general fact is the following result for which we include a proof for the sake of completeness. 

\begin{proposition}\label{complemented c0,I}\cite[p. 2]{kania}
    If $\mathcal I$ is an intersection of finitely many maximal ideals on $\mathbb N$, then $c_{0,\mathcal I}$ is complemented in $\ell_\infty$ and, in particular, it is Grothendieck.
\end{proposition}

\begin{proof}
Let $\mathcal J_1,\ldots,\mathcal J_n$ be maximal ideals on $\mathbb N$ such that $\mathcal I=\bigcap_{k=1}^n\mathcal J_k$. Consider the map $\Phi\colon\ell_\infty\to\mathbb R^n$ defined by
\begin{gather*}
     \Phi({\bf x})=(\mathcal J_1^*-\lim x_n,\ldots,\mathcal J_n^*-\lim x_n), \quad {\bf x}=(x_n).
\end{gather*}
It is not difficult to check that $\Phi$ is an onto bounded linear operator. Note that $\ker\Phi=c_{0,\mathcal I}$. So,
$\ell_\infty/c_{0,\mathcal I}$ is finite dimensional. Hence, $c_{0,\mathcal I}$ is complemented in $\ell_\infty$.
\end{proof}


\begin{theorem}
If $\ideal$ is an intersection of finitely many maximal ideals on $\mathbb N$, then $c_{0,\ideal^\omega}$ is complemented in $\ell_\infty$
and, in particular,  is Grothendieck.
\end{theorem}

\proof From the hypothesis, we know that $c_{0,\mathcal I}$ is complemented in $\ell_\infty$ (Proposition \ref{complemented c0,I}). Let $Y\subset\ell_\infty$ be a subspace such that $\ell_\infty=c_{0,\mathcal I}\oplus Y$. So, 
\begin{gather*}
    \ell_\infty\sim\ell_\infty(\ell_\infty)\sim\ell_\infty(c_{0,\mathcal I})\oplus_\infty\ell_\infty(Y)\sim c_{0,\mathcal I^{\omega}}\oplus_\infty\ell_\infty(Y),
\end{gather*}
by Theorem \ref{directsum}. Whence $c_{0,\ideal^\omega}$ is isomorphic to a complemented subspace of $\ell_\infty$. 
\endproof

\begin{question}
 Is $c_{0,{\sf WO}}$ a Grothendieck space?
\end{question}

\section*{Acknowledgments}

Part of the research of this paper was developed during a postdoctoral stay of the first author supported by Funda\c c\~ao de Apoio \`a Pesquisa do Estado de S\~ao Paulo, FAPESP, Processo 2021/01144-2, and UIS.


\begin{thebibliography}{99}

\bibitem {AK}  Albiac, Fernando; Kalton, Nigel J. \textit{Topics in Banach space theory.} Second edition. With a foreword by Gilles Godefory. Graduate Texts in Mathematics, 233. Springer, [Cham], 2016. 


\bibitem{balcerzak} Balcerzak, Marek; Pop\l awski, Micha\l; Wachowicz, Artur. \textit{The Baire category of ideal convergent subseries and rearrangements.} Topology Appl. 231 (2017), 219-230.

\bibitem{BFMS2013}
Barbarski, P, Filip\'ow, R.,  Mrozek, N, and Szuca, P. {\em When does the Kat\v{e}tov order imply that one ideal extends the other?} Colloq. Math., 130(1):91–102, 2013. 


\bibitem{boroludin et all} Borodulin-Nadzieja, Piotr; Farkas, Barnabás; Plebanek, Grzegorz.
\textit{Representations of ideals in Polish groups and in Banach spaces}.
J. Symb. Log. 80 (2015), no. 4, 1268–1289. 


\bibitem{cembranos-mendoza2010} Cembranos, Pilar and Mendoza, Jos\'{e}.
\textit{The {B}anach spaces {$\ell_\infty(c_0)$} and  {$c_0(\ell_\infty)$} are not isomorphic},
{J. Math. Anal. Appl.}, 367(2), 461-463 (2010).
		

\bibitem{demirci} Demirci, Kamil. \textit{$I$-limit superior and limit inferior.} Math. Commun. 6 (2001), no. 2, 165-172.

\bibitem{diestel and all} Diestel, Joe; Fourie, Jan H.; Swart, Johan. \textit{The metric theory of tensor products. Grothendieck's résumé revisited.} American Mathematical Society, Providence, RI, 2008.



\bibitem{fast} Fast, H. \textit{Sur la convergence statistique}, Colloquium Math. 2 (1951), 241–244 (1952).

\bibitem{filipow} Filip\'ow, Rafa\l; Mrozek, Nikodem; Rec\l aw, Ireneusz; Szuca, Piotr. \textit{Ideal convergence of bounded sequences.} J. Symbolic Logic 72 (2007), no. 2, 501–512.

\bibitem{galego} Galego, El\'oi Medina. \textit{On the structure of into isomorphisms between spaces of continuous functions.} Proc. Amer. Math. Soc. 151 (2023), no. 2, 693–706.

\bibitem{gonzalez-kania} Gonz\'alez, Manuel; Kania, Tomasz.\textit{ Grothendieck spaces: the landscape and perspectives.} Japan. J. Math. 16 (2021), no. 2, 247-313.

\bibitem{GU2018}
Guevara, F. and Uzc\'ategui, C.
\newblock {\em Fr\'echet {B}orel ideals with {B}orel orthogonal}.
\newblock {Colloquium Mathematicum}, 152(1):141--163 (2018).

\bibitem{holsz} Holszty\'nski W., \textit{Continuous mappings induced by isometries of spaces of continuousfunctions}. Studia Math.26 (1966), 133–136.

\bibitem{Hrusak2011}
Hru{\v{s}}{\'a}k, M.
\newblock Combinatorics of filters and ideals.
\newblock In {\em Set theory and its applications}, volume 533 of {\em Contemp.
  Math.}, pages 29--69. Amer. Math. Soc., Providence, RI, 2011.

\bibitem{Hrusak2017}
Hru{\v{s}}{\'a}k, M.
\newblock Kat{\v{e}}tov order on {B}orel ideals.
\newblock {\em Archive for Mathematical Logic}, 56(7-8):831--847, 2017.


\bibitem{kania} Kania, Tomasz.
\textit{A letter concerning Leonetti's paper `Continuous projections onto ideal convergent sequences'.} 
Results Math. 74 (2019), no. 1, 4 pp.

\bibitem{k-s-w} Kostyrko, P., Salat, T., Wilczy\'nski, W. \textit{$\mathcal I$-convergence}, Real Anal. Ex-
change 26 (2000-2001), 669-686.

\bibitem{k-m-s} Kostyrko, P., Macaj, M., Salat, T., Sleziak, M. $\mathcal I$-convergence and extremal $\mathcal I$-limit points. Math. Slovaca 55(4), 443-464 (2005).

\bibitem{Leonetti2018} Leonetti, P.
{\em Continuous projections onto ideal convergent sequences}. Results Math., 73 (3): 114,5 (2018).

\bibitem{leung-rabiger} Leung, D.  and  R\"abiger, F. \textit{Complemented copies of $c_0$ in $l_\infty$ sums of Banach spaces}.
Illinois J. Math. 34 (1990), 52–58.

\bibitem{luxe} Luxemburg, W.A.J., Zaanen, A.C.: \textit{Riesz spaces I.} North-Holland, Amsterdam (1971)

\bibitem{rincon-villamizar} Rinc\'on-Villamizar, Michael A.
\textit{A proof of Holszty\'nski theorem.}
Rev. Integr. Temas Mat. 36 (2018), 1, 59–65.


\bibitem{To} Todor\v{c}evi\'{c}, S. Analytic gaps. {\em Fundamenta
Mathematicae}. 150 (1996), pp. 55-66.

\bibitem{uzcasurvey}
Uzc\'{a}tegui-Aylwin, C.
\newblock Ideals on countable sets: a survey with questions.
\newblock {\em Rev. Integr.temas mat.}, 37(1):167--198, 2019.


\end{thebibliography}
\end{document}